\documentclass[11pt]{amsart}
\usepackage[latin1]{inputenc}
\usepackage{amsmath}
\usepackage{amsfonts}
\usepackage{amssymb}
\usepackage{graphicx}
\usepackage{fourier}
\usepackage{hyperref}
\usepackage{enumerate}
\usepackage{esint}
\usepackage{bm}
\usepackage{ulem}
\usepackage{xcolor} 
\usepackage{verbatim}        

\usepackage[english]{babel}
\newtheorem{theorem}{Theorem}[section]
\newtheorem{lemma}[theorem]{Lemma}
\newtheorem{prop}[theorem]{Proposition}

\theoremstyle{definition}
\newtheorem{definition}[theorem]{Definition}

\newtheorem{corollary}[theorem]{Corollary}
\theoremstyle{remark}
\newtheorem{remark}[theorem]{Remark}

\newcommand{\inner}[1]{\left\langle#1\right\rangle}
\newcommand{\norm}[1]{\left\lVert#1\right\rVert}
\newcommand{\abs}[1]{\left\lvert#1\right\rvert}
\newcommand{\pa}[1]{\left( #1 \right)}
\newcommand{\rpa}[1]{\left[ #1 \right]}
\newcommand{\br}[1]{\left\lbrace #1\right\rbrace}

\newcommand{\R}{\mathbb{R}}
\newcommand{\CC}{\mathbb{C}}
\newcommand{\Z}{\mathbb{Z}}
\newcommand{\N}{\mathbb{N}}
\newcommand{\T}{\mathbb{T}}

\newcommand{\F}{\mathcal{F}}
\newcommand{\D}{\mathcal{D}}
\newcommand{\PP}{\mathbf{P}}

\numberwithin{equation}{section}

\newcommand{\ad}{\partial_x^{-1}}
\newcommand{\smin}{\sigma_{\mathrm{min}}}
\newcommand{\smax}{\sigma_{\mathrm{max}}}

\definecolor{bostonuniversityred}{rgb}{0.8, 0.0, 0.0}
 
\definecolor{byzantium}{rgb}{0.44, 0.16, 0.39}

\newcommand{\II}{{\bf I}}
\newcommand{\C}{{\bf C}}

\renewcommand{\span}{\operatorname{span}}

\renewcommand{\Re}{\operatorname{Re}}
\renewcommand{\Im}{\operatorname{Im}}

\title{Large time convergence of the non-homogeneous Goldstein-Taylor Equation}
\author{Anton Arnold, Amit Einav, Beatrice Signorello \& Tobias W\"ohrer}
\address{$^1$Vienna University of Technology, Institute of Analysis and Scientific Computing, Wiedner Hauptstr. 8-10, A-1040 Wien, Austria}
\email{anton.arnold@tuwien.ac.at; beatrice.signorello@tuwien.ac.at}
\email{tobias.woehrer@tuwien.ac.at}
\address{$^2$Karl Franzens Universit\"at Graz, Institute of Mathematics and Scientific Computing, Heinrichstra\ss e 36, 8010, Graz, Austria}
\email{amit.einav@uni-graz.at}
\thanks{The authors gratefully acknowledge the support of the Hausdorff Research Institute for Mathematics
(Bonn), through the Junior Trimester Program on Kinetic Theory. The first and third authors  were  partially  supported  by  the  FWF  (Austrian Science Fund) funded SFB \#F65, the third and forth authors by the FWF-doctoral school W 1245.}

\begin{document}

\maketitle
\begin{abstract}
The Goldstein-Taylor equations can be thought of as a simplified version of a BGK system, where the velocity variable is constricted to a discrete set of values. It is intimately related to turbulent fluid motion and the telegrapher's equation. A detailed understanding of the large time behaviour of the solutions to these equations has been mostly achieved in the case where the relaxation function, measuring the intensity of the relaxation towards equally distributed velocity densities, is constant. The goal of the presented work is to provide a general method to tackle the question of convergence to equilibrium when the relaxation function is not constant, and to do so as quantitatively as possible. In contrast to the usual modal decomposition of the equations, which is natural when the relaxation function is constant, we define a new Lyapunov functional of pseudodifferential nature, one that is motivated by the modal analysis in the constant case, that is able to deal with full spatial dependency of the relaxation function. The approach we develop is robust enough that one can apply it to multi-velocity Goldstein-Taylor models, and achieve explicit rates of convergence. The convergence rate we find, however, is not optimal, as we show by comparing our result to the that found in \cite{BS}.
\end{abstract}
{\tiny
{KEYWORDS.}
BGK equation, hypocoercivity, large time behaviour, exponential decay, Lyapunov functional}
\\{\tiny{MSC.}
82C40 (Kinetic theory of gases in time-dependent statistical mechanics), 35B40 (Asymptotic behavior of solutions to PDEs), 35Q82 (PDEs in connection with statistical mechanics), 35S05 (Pseudodifferential operators as generalizations of partial differential operators)}

\section{Introduction}\label{sec:intro}

The object of this work is the large time analysis of the Goldstein-Taylor equations on the one-dimensional torus $\T$, i.e.\ on $[0,2\pi]$ with periodic boundary conditions, and for $t\in(0,\infty)$:
\begin{equation}\label{eq:gt}
\begin{gathered}
\partial_t f_+(x,t) + \partial_x f_+(x,t) = \frac{\sigma(x)}{2}(f_-(x,t)-f_+(x,t)),\\
\partial_t f_-(x,t) - \partial_x f_-(x,t) = - \frac{\sigma(x)}{2}(f_-(x,t)-f_+(x,t)),\\
f_{\pm}(x,0)= f_{\pm,0}(x),
\end{gathered}
\end{equation}
where $f_{\pm}(x,t)$ are the density functions of finding an element with a velocity $\pm 1$ in a position $x\in\mathbb{T}$ at time $t>0$. The function $\sigma\in L^{\infty}_+\pa{\mathbb{T}}:=\big\{f\in L^\infty(\T)\;\big|\;\mbox{essmin }f>0\big\}$ is the relaxation coefficient, and $f_{\pm,0}$
are the initial conditions. 
Since \eqref{eq:gt} is mass conserving, its steady state is of the form
$$
  f_{\pm,\infty}(x) := f_\infty\ ,\quad x\in\T\ ;\qquad f_\infty:=\frac12 (f_{+,0}+f_{-,0})_{\mathrm{avg}},
$$
with the notation
\begin{equation}\label{eq:avg}
h_{\mathrm{avg}}:=\frac{1}{2\pi}\int_{0}^{2\pi} h(x)dx.
\end{equation}

The Goldstein-Taylor model was originally considered as a diffusion process, resulting as a limit of a discontinuous random migration in 1D, where particles may change direction with rate $\sigma$. It appeared in the context of turbulent fluid motion and the telegrapher's equation, see \cite{T22, G51}, respectively. 
\eqref{eq:gt} can also be seen as a special 1D case of a BGK-model (named after the three physicists Bhatnagar, Gross, and Krook \cite{BGK}) with a discrete set of velocities. Such equations commonly appear in applications like gas and fluid dynamics as velocity
discretisations of various kinetic models (e.g.\ the Boltzmann equation). The mathematical
analysis of such discrete velocity models has a long standing
tradition, see \cite{CIS87,Kaw91} and references therein.

Although the Goldstein-Taylor equation is very simple, it still exhibits an interesting and mathematically rich structure. Hence, it has been attracting continuous interest over the last 20 years. Most of its mathematical analyses was devoted to the following three topics: scaling limits, asymptotic preserving (AP) numerical schemes, and large time behaviour. In a diffusive scaling, the Gold\-stein-Taylor model can be viewed as a hyperbolic approximation to the heat equation \cite{Sa99}. Various AP-schemes for this model in the stiff relaxation regime (i.e.\ for $\sigma\to\infty$) were constructed and analyzed in \cite{Jin99, GT02, AHJP14}.
Since the large time convergence of solutions to \eqref{eq:gt} towards its unique steady state is also the topic of this work, we shall review the related literature in more detail:

Analytically, the main difficulty of \eqref{eq:gt} is with its hypocoercivity, as defined in \cite{Villani-book}: More specifically, the relaxation operator on the r.h.s.\ is not coercive on $\T\times \R^2$. Hence, for each fixed $x$, the r.h.s.\ by itself would drive the system to its local equilibrium, generated by the kernel of the relaxation operator, $\span\{\binom{1}{1}\}$, but the local mass (density) might be different at different positions. Convergence to the global equilibrium $(f_\infty,f_\infty)^T$ only arises due to the interplay between local relaxation and the transport operator on the l.h.s.\ of \eqref{eq:gt}.

The Goldstein-Taylor model was also considered in the analysis of \cite{ACT02}, if one chooses the velocity matrix to be $V=\mbox{diag}(1,-1)$ and the relaxation matrix $\bm{A}(x)$ to be
$$\bm{A}(x) = \frac{\sigma(x)}{2} \begin{pmatrix}
1& -1\\
-1& 1
\end{pmatrix}\ge 0.$$
Exponential convergence to the steady state is then proved in the aforementioned work for the system \eqref{eq:gt} \textit{with inflow boundary conditions}. Such boundary conditions make the problem significantly easier than in the periodic set-up envisioned here, in particular it allows for $\sigma(x)$ to be zero on a subset of $\T$, an issue that proves to be far more difficult in our setting.\\ 
In \cite{DS09} the authors proved polynomial decay towards the equilibrium, allowing $\sigma(x)$ to vanish at finitely many points.\\
In \cite{Tr13} the author proved exponential decay for solutions to \eqref{eq:gt} for a more general $\sigma(x)\ge0$. That work is based on a (non-local in time) {\it weak coercive estimate} on the damping.\\
All of the papers mentioned so far did not focus on the optimality of the (exponential) decay rate. Using the equivalence between \eqref{eq:gt} and the telegrapher's equation, the authors of 
\cite{BS} have shown that this optimal decay rate, $\mu(\sigma)$, is the minimum of $\sigma_{\mathrm{avg}}$ and the spectral gap of the telegrapher's equation (excluding the case when some of those eigenvalues with real part equal to $\mu(\sigma)$ are defective). The precise value of this spectral gap, however, is hardly accessible - even for simple non-constant relaxation functions $\sigma(x)$ (see e.g.\ Appendix \ref{appsec:lack}). Moreover, it is based on the restrictive requirement $f_{\pm,0}\in H^1(\T)$, and cannot be extended to other discrete velocity models in 1D. The reason for the latter is that \cite{BS} heavily relies on the equivalence of \eqref{eq:gt} to the telegrapher's equation.

The issues above motivated our subsequent analysis: We {introduce a method for $L^2$--initial data} that can be extended to other discrete velocity BGK-models (as illustrated below for a $3-$velocities system), and that yields sharp rates for constant $\sigma$. Moreover, and most importantly, it {is applicable in the} general non-homogeneous $\sigma\in L^{\infty}_+\pa{\mathbb{T}}$ {case} and yields in these cases an explicit, quantitative lower bound for the decay rate. In {this} case, however, it will not achieve an optimal rate of convergence\footnote{at least compared to the $H^1$-result in \cite{BS}} to the appropriate equilibrium of the system.  
The method to be derived here will use a Lyapunov function technique in the spirit of the earlier works \cite{Villani-book, DMS, AAC, Achleitner2018}. \\

This paper is structured as follows: In \S\ref{subsec:main} we give the analytical setting of the problem and present our main convergence result (Theorem  \ref{thm:main}). 
In \S\ref{sec:pre} we recall some analytical results which will be needed in the analysis that will follow, and explore some properties of the entropy functional $E_\theta$ and the anti-derivative of functions on $\T$, defined in \eqref{eq:def_entropy} and \eqref{eq:antiderivative}, respectively.  \S \ref{sec:sigma_constant} is devoted to the case where $\sigma(x)=\sigma$ is constant, which will motivate our more general approach: Based on a modal decomposition of the Goldstein-Taylor system and its spectral analysis we derive the entropy functional $E_\theta$, first on a modal level and then as a pseudo-differential operator in physical space. We conclude by proving part \eqref{item:convergence_sigma_constant} of our main theorem. Continuing to  \S\ref{sec:sigma_general}, we will prove, using a perturbative approach to the problem, part \eqref{item:convergence_sigma_not_constant} of our main theorem. The robustness of our method will be shown in \S\ref{sec:three-velocities} where we use it to obtain an explicit rate of convergence for a $3-$velocities Goldstein-Taylor model. Finally, in Appendix \ref{appsec:lack} we discuss a potential way to improve the technique from \S\ref{sec:sigma_general}, and explicitly show the lack of optimality of it for a particular case {of $\sigma(x)$}.


\section{The setting of the problem and main results}\label{subsec:main} To better understand the Goldstein-Taylor system, \eqref{eq:gt}, one starts by recasting it in the macroscopic variables
$$
  u: = f_+ + f_-,\quad\quad v:=f_+-f_-,
$$
representing the spatial (mass) density and the {flux} density, respectively.
The macroscopic variables yield the following system of equations on $\mathbb{T}\times (0,\infty)$:
\begin{equation}\label{eq:gt_recast}
\begin{gathered}
\partial_t u(x,t) +\partial_x v(x,t) =0,\\
\partial_t v(x,t) +\partial_x u(x,t) =- \sigma(x)v(x,t),\\
u(.,0)=u_0:=f_{+,0}+f_{-,0},\quad v(.,0)=v_0:=f_{+,0}-f_{-,0}\ ,
\end{gathered}
\end{equation}
whose theory of existence and uniqueness is straightforward (since the r.h.s.\ is a bounded perturbation of the transport operator; see \S2 in \cite{DS09} or, more generally, \cite{Pa92}). Moreover, when one tries to understand the qualitative behaviour of \eqref{eq:gt_recast}, one notices that the equation for $u$ speaks of ``total mass conservation'' (upon integration over the spatial interval $(0,2\pi)$), while the equation for $v$ predicts a strong decay to zero for the function. This means, at least intuitively, that the difference between $f_+$ and $f_-$ should go to zero, and that their sum retains its mass. As the main driving force of the equation is a transport operation on the torus, we will not be surprised to learn that the large time behaviour of $u$ (and since $v$ should go to zero, of $f_+$ and $f_-$ as well) is convergence to a constant. All of this has been verified in several cases, most generally in \cite{BS}.

We now set the framework that will assist us in the investigation of the large time behaviour of \eqref{eq:gt_recast}, in a relatively general case.
The natural Hilbert space to consider this problem is $L^2(\T)^{\otimes 2}$, with the standard inner product for each component: 
$$\inner{f_1,f_2}:=\frac{1}{2\pi}\int_{0}^{2\pi} f_1(x)\overline{f_2(x)}dx,$$
where the bar denotes complex conjugation. 
Since \eqref{eq:gt} and \eqref{eq:gt_recast} are (only) hypocoercive, the symmetric part of their generators (i.e.\ the operators on their r.h.s.) are not coercive on $L^2(\T)^{\otimes 2}$. Hence, the standard $L^2$--norm cannot serve as a usable Lyapunov functional. As is typical for hypocoercive equations (see \cite{Villani-book, DMS, AAC}), a possible remedy to this problem is to consider a ``twisted'' norm (often also referred to as {\it entropy functional}), constructed in a way that this functional strictly decays along each trajectory $(u(t),v(t))$. 
\\ The following functional, which will be our entropy functional, is not an ansatz, and its origin will be derived in \S\ref{sec:sigma_constant}. Moreover, we will show that it will yield the sharp exponential decay for constant $\sigma$, when one chooses $\theta=\theta(\sigma)$ appropriately.

\begin{definition}\label{def:entropy}
Let $f,g\in L^2\pa{\T}$ and let $\theta>0$ be given. Then we define the {\it entropy} $E_\theta(f,g)$ as
\begin{equation}\label{eq:def_entropy}
E_\theta(f,g):=\norm{f}^2+\norm{g}^2 -\frac{\theta}{2\pi}\int_{0}^{2\pi} \Re\pa{\partial_x^{-1}f (x)\overline{g(x)}}dx.
\end{equation}
Here, the {\it anti-derivative} of $f$ is defined as
\begin{equation}\label{eq:antiderivative}
\partial_x^{-1}f (x):=\int_{0}^{x}f(y)dy - \pa{\int_{0}^{x}f(y)dy }_{\mathrm{avg}}\ ,
\end{equation}
with the average defined in \eqref{eq:avg}.
The normalization constant in \eqref{eq:antiderivative} is chosen such that $(\partial_x^{-1}f)_{\mathrm{avg}}=0$.
\end{definition}

Several recent studies (like \cite{DMS, AAC}) considered the Goldstein-Taylor system with constant $\sigma$. This case can be investigated fairly easy as one is able to utilise Fourier analysis in this setting, and construct a Lyapunov functional as a sum of quadratic functionals of the Fourier modes. However, the moment we change $\sigma(x)$ to a non-constant function - even to one that is natural in the Fourier setting, such as sine or cosine - the Fourier analysis becomes nigh impossible to solve. \\
The main idea that guided us in our approach was to re-examine the case where $\sigma$ is constant and \textit{to recast the modal Fourier norm by using a pseudo-differential operator}, without needing its modal decomposition. This functional, which is exactly $E_\theta$ for particular choices of $\theta=\theta(\sigma)$, can then be \textit{extended} to the case where $\sigma(x)$ is not constant, yielding quantitative estimates for the convergence. As the nature of this approach is perturbative, our decay rates are not optimal. The methodology itself, however, is fairly robust, and is viable in other cases, such as the multi-velocity Goldstein-Taylor model (as we shall see).

The main theorem we will show in this paper, with the use of the vector notation
\begin{equation}\label{vector-not}
  f(t):=\binom{f_+(t)}{f_-(t)}\ ,\quad f_0:=\binom{f_{+,0}}{f_{-,0}}\ ,
\end{equation}
is the following:
\begin{theorem}\label{thm:main}
Let $u,v\in C([0,\infty);L^2\pa{\T})$ be mild\footnote{We use {\it mild solution} in the termonology of semigroup theory \cite{Pa92}.} real valued solutions to \eqref{eq:gt_recast} with initial datum $u_0,\, v_0\in L^2\pa{\T}$. Denoting by $u_{\mathrm{avg}}=\pa{u_0}_{\mathrm{avg}}$ we have:
\begin{enumerate}[a)]
\item\label{item:convergence_sigma_constant} \underline{If $\sigma(x)=\sigma$ is constant} we have that:\\
If $\sigma\not=2$ then
$$E_{\theta(\sigma)}\pa{u(t)-u_{\mathrm{avg}},v(t)} \leq E_{\theta(\sigma)}\pa{u_0-u_{\mathrm{avg}},v_0}e^{-2\mu\pa{\sigma} t}$$
where 
$$\theta\pa{\sigma}:=\begin{cases} \sigma, & 0<\sigma<2 \\  \frac{4}{\sigma}, & \sigma>2 \end{cases}, \quad \mu\pa{\sigma}:=\begin{cases} \frac{\sigma}{2}, & 0<\sigma<2 \\ \frac{\sigma}{2}-\sqrt{\frac{\sigma^2}{4}-1}, & \sigma>2 \end{cases},$$
and if $\sigma=2$ then for any $0<\epsilon<1$
$$E_{\frac{2\pa{2-\epsilon^2}}{2+\epsilon^2}}\pa{u(t)-u_{\mathrm{avg}},v(t)} \leq E_{\frac{2\pa{2-\epsilon^2}}{2+\epsilon^2}}\pa{u_0-u_{\mathrm{avg}},v_0}e^{-2\pa{1-\epsilon}t}.$$
Consequently if $\sigma \not=2$
\begin{equation}\label{eq:main_for_f_pm_constant}
\begin{gathered}
\Big\|f(t)-\binom{f_\infty}{f_\infty}\Big\| \leq C_{\sigma} \Big\|f_0-\binom{f_\infty}{f_\infty}\Big\| e^{-\mu\pa{\sigma}t},
\end{gathered}
\end{equation}
where
\begin{equation}\label{eq:C_sigma_and_f_infty}
C_{\sigma}:=\begin{cases} \sqrt{\frac{2+\sigma}{2-\sigma}}, & 0<\sigma<2 \\  \sqrt{ \frac{\sigma+2}{\sigma-2}}, & \sigma>2 \end{cases}, \quad f_{\infty}=\frac{u_{\mathrm{avg}}}{2}\ ,
\end{equation}
and the decay rate $\mu(\sigma)$ is sharp.\\
For $\sigma=2$ we have that 
\begin{equation}\label{eq:main_for_f_pm_constant_ep}
\begin{gathered}
\Big\|f(t)-\binom{f_\infty}{f_\infty}\Big\| \leq \frac{\sqrt2}{\epsilon} \Big\|f_0-\binom{f_\infty}{f_\infty}\Big\| e^{-\pa{1-\epsilon}t}\ .
\end{gathered}
\end{equation}
\medskip

\item\label{item:convergence_sigma_not_constant} \underline{If $\sigma(x)$ is non-constant} such that 
$$0<\smin:=\inf_{x\in\T}\sigma(x) < \sup_{x\in\T}\sigma(x)=:\smax<\infty,$$
then by defining
\begin{equation}\label{eq:def_of_theta_ast}
\theta^\ast:=\min\pa{\smin,\frac{4}{\smax}}
\end{equation}
and
\begin{equation}\label{eq:value_of_alpha}
\alpha^\ast:=\alpha^\ast\pa{\smin,\smax}:=\begin{cases}
\frac{\smin\pa{4+2\sqrt{4-\smin^2}-\smin\smax}}{4+2\sqrt{4-\smin^2}-\smin^2},  & \smin < \frac{4}{\smax}\\
\smax - \sqrt{\smax^2-4}, & \smin \geq  \frac{4}{\smax} 
\end{cases}
\end{equation}
we have that 
$$E_{\theta^\ast}\pa{u(t)-u_{\mathrm{avg}},v(t)} \leq E_{\theta^\ast}\pa{u_0-u_{\mathrm{avg}},v_0} 
e^{-\alpha^\ast t},$$
and as result
\begin{equation}\label{eq:main_for_f_pm_non_constant}
\begin{gathered}
\Big\|f(t)-\binom{f_\infty}{f_\infty}\Big\| 
\leq \sqrt{\frac{2+\theta_\ast}{2-\theta_\ast}} 
\Big\|f_0-\binom{f_\infty}{f_\infty}\Big\| 
e^{-\frac{\alpha^\ast}{2}t}\ ,
\end{gathered}
\end{equation}
with $f_\infty$ defined in \eqref{eq:C_sigma_and_f_infty}.
\end{enumerate}
\end{theorem}
\noindent
Part (a) of this theorem will be proved in \S\ref{subsec:entropy-evol}, and Part (b) in \S\ref{sec:sigma_general}. {In many of the proofs which will eventually lead to the proof of this theorem we will assume that $(u,v)$ is a classical solution, pertaining to $u_0,\; v_0$ in the periodic Sobolev space $H^1(\T)$. The general result will follow by a simple density argument.}

\begin{remark}\label{rem:limit_of_parameters}
It is simple to see that if $\sigma(x)$ satisfies the conditions of \eqref{item:convergence_sigma_not_constant}, then, as $\smin$ and $\smax$ approach a positive constant $\sigma\not=2$, we find that
$$
  \theta^\ast \to \min\pa{\sigma,\frac{4}{\sigma}},\quad \text{and}\quad \alpha^\ast \to \begin{cases} \sigma-\sqrt{\sigma^2-4}, & \sigma >2 \\ \sigma, & \sigma<2\end{cases}\ ,
$$
recovering the results of part \eqref{item:convergence_sigma_constant} of the above theorem.\\
In addition, one should note that when $\smin > \frac{4}{\smax} $ we have that 
$$\alpha^\ast\pa{\smin,\smax}=2\mu\pa{\smax},$$
where $\mu\pa{\sigma}$ was defined in part \eqref{item:convergence_sigma_constant} of the Theorem. This validates the intuition that, if $\smax$ is ``large enough'', the convergence rate of the solution can be estimated using the ``worst convergence rate'', corresponding to $\mu\pa{\smax}$ of the $\sigma(x)=\sigma$ case.\\
Lastly, one notices that when $\smin=\frac{4}{\smax}$
$$\frac{\smin\pa{4+2\sqrt{4-\smin^2}-\smin\smax}}{4+2\sqrt{4-\smin^2}-\smin^2} =\smax - \sqrt{\smax^2-4},$$
which shows the continuity of $\alpha^{\ast}$ on the curve that stitches the two formulas in \eqref{eq:value_of_alpha}.
\end{remark}

\section{Preliminaries}\label{sec:pre}
In this short section we will remind the reader of a few simple properties of functions on the torus, as well as explore properties of the anti-derivative function, $\partial_x^{-1} f$, and our functional $E_{\theta}(f,g)$. Most of the simple proofs of this section will be deferred to Appendix \ref{appsec:proofs}.

We begin with the well known Poincar\'e inequality:
\begin{lemma}[Poincar\'e Inequality]\label{thm:poincare}
Let $f\in H^1_{per}\pa{\T}$ with $f_{\mathrm{avg}}=0$. Then
\begin{equation}\label{eq:poincare}
\norm{f}\leq \norm{f^\prime}.
\end{equation}
\end{lemma}

Next we focus our attention on some simple, yet crucial, properties of the anti-derivative function which was defined in \eqref{eq:antiderivative}.
\begin{lemma}\label{lem:properties_of_antiderivative}
Let $f\in L^1\pa{\T}$. Then:
\begin{enumerate}[i)]
\item\label{item:avg} $\pa{\partial^{-1}_x f}_{\mathrm{avg}}=0$.
\item\label{item:diff} $\partial_x^{-1}f$ is differentiable a.e.\ on $[0,2\pi]$ and $\partial_x\pa{\partial^{-1}_x f}(x)=f(x)$ a.e.
\item\label{item:anti_of_deri} If in addition $f$ is differentiable we have that $\partial^{-1}_x\pa{\partial_x f}(x)=f(x)-f_{\mathrm{avg}}$.
\item\label{item:cont_and_fourier_coefficients} If, in addition, we have that $f_{\mathrm{avg}}=0$, then $\partial_x^{-1}f$ is a continuous function on the torus, and
\begin{equation}\label{eq:fourier_coefficients_of_antiderivative}
\widehat{\partial^{-1}_x f}\pa{k}=
= \left\{
\begin{array}{ll}
\frac{\widehat{f}(k)}{ik}, & \: k\ne 0 \\
0, &  \: k=0 \\
\end{array}
\right. \ .
\end{equation}
\end{enumerate}
\end{lemma}

\begin{remark}\label{rem:integration_by_parts}
{\eqref{item:diff}, \eqref{item:cont_and_fourier_coefficients}, and the fact that $f$ is a function on the torus, imply that if} $f_{\mathrm{avg}}=0$ we are allowed to use integration by parts with $\partial_x^{-1}f(x)$ on this boundaryless manifold without qualms. 
\end{remark}
The last simple lemma in this revolves around our newly defined functional, $E_{\theta}$.

\begin{lemma}\label{lem:entropy_equivalence}
Let $f,g\in L^2\pa{\T}$ be such that $f_{\mathrm{avg}}=0$ and let $\theta\in \R$ be given. Then the entropy $E_\theta(f,g)$, defined in \eqref{eq:def_entropy}, satisfies
\begin{equation}\label{eq:entropy_equivalence_always}
E_\theta\pa{f,g}\leq \pa{1+\frac{\abs{\theta}}{2}}\pa{\norm{f}^2+\norm{g}^2}\ .
\end{equation}
If in addition $\abs{\theta}<2$ we have that 
\begin{equation}\label{eq:entropy_equivalence_restricted}
E_\theta\pa{f,g}\geq \pa{1-\frac{\abs{\theta}}{2}}\pa{\norm{f}^2+\norm{g}^2}\ .
\end{equation}
In particular, if $0\leq \theta<2$ we have that 
\begin{equation}\label{eq:entropy_equivalence}
\pa{1-\frac{\theta}{2}}\pa{\norm{f}^2+\norm{g}^2}\leq E_\theta\pa{f,g}\leq \pa{1+\frac{\theta}{2}}\pa{\norm{f}^2+\norm{g}^2}\ .
\end{equation}
\end{lemma}

Lastly, we shall prove the following theorem, which (finally) brings the system \eqref{eq:gt_recast} into play, and on which we will rely on frequently in our future estimation.

\begin{prop}\label{thm:evolution_of_entropy}
Let $u,v\in C([0,\infty);L^2\pa{\T})$ be (real valued) mild solutions to \eqref{eq:gt_recast} with initial datum $u_0,\; v_0\in L^2\pa{\T}$. Then for any $\theta\in\R$
\begin{equation}\label{eq::evolution_of_entropy}
\begin{gathered}
\frac{d}{dt}E_{\theta}\pa{u(t)-u_{\mathrm{avg}},v(t)}=-\theta \norm{u(t)-u_{\mathrm{avg}}}^2 + \frac{1}{2\pi} \int_0^{2\pi}(\theta-2\sigma(x))v(x,t)^2 dx\\
+\frac{\theta}{2\pi}\int_0^{2\pi}\sigma(x) \partial_x^{-1}  \pa{u(x,t)-u_{\mathrm{avg}}}  v(x,t)dx-\theta \pa{v(t)_{\mathrm{avg}}}^2\ ,
\end{gathered}
\end{equation}
where
\begin{equation}\label{eq:u_avg}
u_{\mathrm{avg}}=\frac{1}{2\pi}\int_{0}^{2\pi}u_0(x)dx=\frac{1}{2\pi}\int_{0}^{2\pi}u(x,t)dx,\quad \forall t>0.
\end{equation}
\end{prop}
\begin{proof}
We begin by noticing that the validity of \eqref{eq:u_avg} follows immediately from the fact that $u$ is a mild  solution and the conservation of mass property of the system \eqref{eq:gt_recast}. Moreover, one can see that replacing $\pa{u(t),v(t)}$ by $\pa{u(t)-u_{\mathrm{avg}},v(t)}$ yields an equivalent solution (up to a constant shift in the initial data) to the system of equations, with the additional condition that the average of the first component is zero for all $t\geq 0$. With this observation in mind, we can assume without loss of generality that $u_{\mathrm{avg}}=0$.

Using the Goldstein-Taylor equations we see that
$$\frac{d}{dt}\norm{u(t)}^2=2\inner{u,\partial_t u}=-2\inner{u,\partial_x v}.$$
$$\frac{d}{dt}\norm{v(t)}^2=2\inner{v,\partial_t v}=-2\inner{v,\partial_x u+\sigma v}.$$
Since
$$\inner{u,\partial_x v}+\inner{v,\partial_x u}=\frac{1}{2\pi}\int_{0}^{2\pi}\partial_x\pa{uv}(x,t)dx=0\ ,$$
we see that 
\begin{equation}\label{eq:diff_of_norms}
\frac{d}{dt}\pa{\norm{u(t)}^2+\norm{v(t)}^2}=-\frac{1}{\pi}\int_{0}^{2\pi}\sigma(x)v(x,t)^2dx.
\end{equation}
We now turn our attention to the mixed term of $E_\theta(u,v)$:
\begin{eqnarray*}
&&\!\!\!\!\!\!\!\!\!\!\!\!\frac{d}{dt}\frac{\theta}{2\pi}\int_{0}^{2\pi} \ad u (x,t)v(x,t)dx\\
&&\!\!\!\!\!\!\!\!\!\!\!\!=\frac{\theta}{2\pi}\int_{0}^{2\pi} \ad \pa{\partial_t u }(x,t)v(x,t)dx
+\frac{\theta}{2\pi}\int_{0}^{2\pi} \ad u (x,t)\partial_t v(x,t)dx\\
&&\!\!\!\!\!\!\!\!\!\!\!\!=-\frac{\theta}{2\pi}\int_{0}^{2\pi} \ad \pa{\partial_x v }(x,t)v(x,t)dx
-\frac{\theta}{2\pi}\int_{0}^{2\pi} \ad u (x,t)[\partial_x u(x,t)+\sigma(x)v(x,t)]dx.
\end{eqnarray*}
Using points \eqref{item:diff} and \eqref{item:anti_of_deri} of Lemma \ref{lem:properties_of_antiderivative}, together with Remark \ref{rem:integration_by_parts}, we find that the above equals
$$-\frac{\theta}{2\pi}\int_{0}^{2\pi} \pa{v(x,t)-v(t)_{\mathrm{avg}}}v(x,t) dx + \frac{\theta}{2\pi}\int_{0}^{2\pi} u(x,t)^2 dx$$
$$-\frac{\theta}{2\pi}\int_{0}^{2\pi} \sigma(x) \ad u(x,t)v(x,t)dx.$$
Subtracting this from \eqref{eq:diff_of_norms} (as there is a minus in definition \eqref{eq:def_entropy}) yields \eqref{eq::evolution_of_entropy}.
\end{proof}


\section{Constant relaxation function}\label{sec:sigma_constant}

In recent years, the investigation of the Goldstein-Taylor model on $\T$ with constant relaxation function $\sigma$ was frequently tackled with a modal decomposition in the Fourier space w.r.t.\ $x$. This approach allows for an extension to other discrete velocity models and even some continuous velocities models \cite{AAC}, but is not suitable for the non-homogeneous case. \\
Before beginning with our investigation we review a few recent results:\\ 
In \cite[\S1.4]{DMS} exponential convergence to equilibrium was shown, but without the sharp rate. In \cite[\S4.1]{AAC} a hypocoercive decay estimate of the form
$$
  \Big\|f(t)-\binom{f_\infty}{f_\infty}\Big\|_{L^2} \le c\,e^{-\mu t} \Big\|f_0-\binom{f_\infty}{f_\infty}\Big\|_{L^2}\ ,
$$
with the vector notation from \eqref{vector-not} 
and the sharp rate 
$$
  \mu(\sigma)= \left\{
\begin{array}{ll}
\sigma, & \:0 < \sigma <2 \\
\frac{\sigma}{2}-\sqrt{\frac{\sigma^2}{4}-1}, &  \:\sigma >2 \\
\end{array}
\right. 
$$
was obtained (see also Fig. \ref{fig:sigma} below). A further study on the minimal constant $c$ in the above was provided in \cite[Th. 1.1]{AAS}.

With these results in mind, we turn our attention to the following (recast) Goldstein-Taylor equation with a constant relaxation rate:
\begin{equation}\label{eq:gt_constant}
\begin{gathered}
\partial_t u(x,t) = -\partial_x v(x,t),\\
\partial_t v(x,t) = -\partial_x u(x,t) - \sigma v(x,t) \ .
\end{gathered}
\end{equation}
In order to be able to discover our entropy functional, we shall consider the straightforward modal analysis in detail. This will allow us to obtain not only explicit decay rates for each Fourier mode, but also an ``optimal Lyapunov functional'' for such given mode, with which we will then be able to construct a non-modal entropy functional in terms of a pseudo-differential operator as defined in \eqref{eq:def_entropy}.\\
As was mentioned in \S\ref{subsec:main}, this 
will give us intuition to the large time behaviour of the equation in several cases even when $\sigma(x)$ is not constant.

\subsection{Fourier analysis and the spectral gap}\label{subsec:Fourier}
One natural way to understand the large time behaviour of \eqref{eq:gt_constant} relies on a simple Fourier analysis \textit{together with} a hypocoercivity technique that was developed by Arnold and Erb in \cite{Arnold2014}. We begin with the former, and focus on the latter from the next subsection onwards. \\
 Using the Fourier transform on the torus (i.e. in the spatial variable), we see that \eqref{eq:gt_constant} is equivalent to infinity many decoupled ODE systems:
 \begin{equation}\label{ODEsystem}
\frac{d}{dt}
\begin{pmatrix}
\widehat{u}(k)\\
\widehat{v}(k)
\end{pmatrix} =  -
\begin{pmatrix}
0& ik\\
ik& \sigma 
\end{pmatrix}\begin{pmatrix}
\widehat{u}(k)\\
\widehat{v}(k)
\end{pmatrix}:=-\C_k \begin{pmatrix}
\widehat{u}(k)\\
\widehat{v}(k)
\end{pmatrix} ,\qquad k\in\Z. 
 \end{equation}
The eigenvalues of the matrices $\C_k\in \CC^{2\times2}$ are given by
$$\lambda_{\pm,k}:= \frac{\sigma}{2} \pm \sqrt{\frac{\sigma^2}{4}-k^2},\quad k\in\Z,$$
and as such:
\begin{itemize}
\item \textit{Invariant space:} For $k=0$ we find that $\lambda_{-,0}=0$ and $\lambda_{+,0}=\sigma$. In fact, as
\begin{equation}\label{C0}
\C_0=\begin{pmatrix}
0& 0\\
0& \sigma 
\end{pmatrix}
\end{equation} 
we can conclude immediately that $\widehat{u}(0,t)=\widehat{u_0}(0)$ and $\widehat{v}(0,t)=\widehat{v_0}(0)e^{-\sigma t}$, corresponding to the mass conservation of the original equation and the rapid decay of the difference between the masses of $f_-$ and $f_+$.
\item \textit{Case I:} For $0<|k|\leq \lfloor\frac{\sigma}{2} \rfloor$, with $\frac{\sigma}{2}\not\in \N$, one finds two real eigenvalues, whose minimum is
$$\lambda_{-,k}=\frac{\sigma}{2} - \sqrt{\frac{\sigma^2}{4}-k^2}=\frac{2k^2}{\sigma + \sqrt{\sigma^2-4k^2}}\ ,$$
i.e. the large time behaviour of $\widehat{u}(k)$ and $\widehat{v}(k)$ is controlled by $e^{-\pa{\frac{\sigma}{2} - \sqrt{\frac{\sigma^2}{4}-k^2}}t}$.
\item \textit{Case II:} For $0<k=\frac{\sigma}{2}\in \N$ the two eigenvalues coincide and are equal to $\frac{\sigma}{2}$. 
Hence, that eigenvalue is defective (i.e. corresponds to a Jordan block of size 2) and the large time behaviour of $\widehat{u}(k)$ and $\widehat{v}(k)$ is controlled by $\pa{1+t}e^{-\frac{\sigma}{2}t}$.
\item \textit{Case III:} For $|k|> \lfloor\frac{\sigma}{2} \rfloor$, one finds two complex conjugate eigenvalues, whose real part equals $\frac{\sigma}{2}$. Thus the large time behaviour of $\widehat{u}(k)$ and $\widehat{v}(k)$ is controlled by $e^{-\frac{\sigma}{2}t}$.
\end{itemize}
 \begin{figure}[ht]
	\centering
	\includegraphics[scale=0.7]{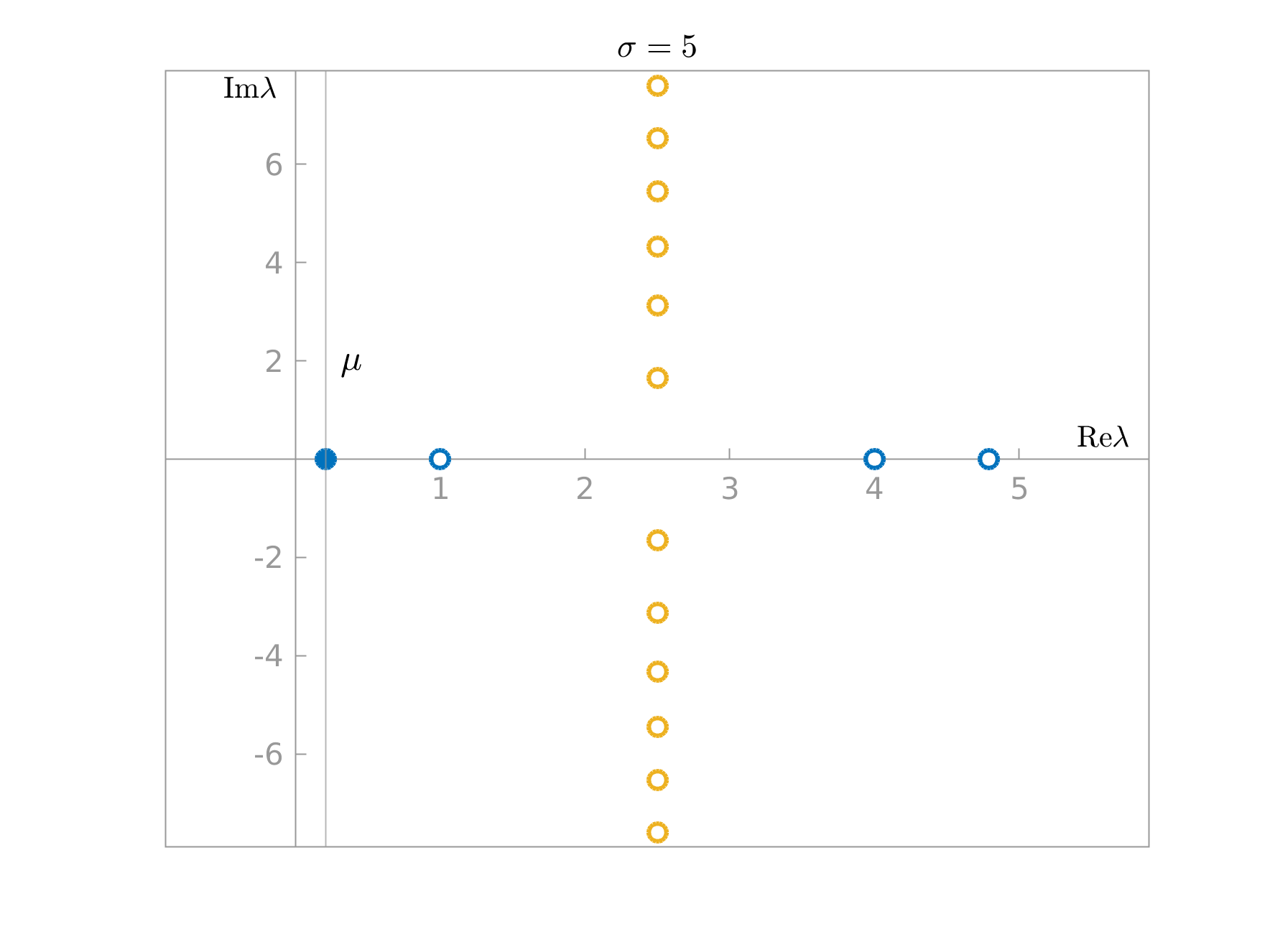}
	\caption{The eigenvalues $\lambda_{\pm,k}$ of $\C_k$, $|k|\in\N$ for $\sigma=5$. The spectral gap is $\mu=(5-\sqrt{21})/2$.}
	\label{fig:eigenvalues_s6}
\end{figure}
From the observations above, we notice that as long as we subtract $\widehat{u}(0)$, i.e. as long as we remove the initial total mass from the original solution, all the modes converge \textit{exponentially} to zero. Their rates have a sharp, and uniform-in-$k$ lower bound that depends on $\sigma$. This spectral gap of  \eqref{eq:gt_constant} will be denoted by $\mu\pa{\sigma}$.\\
Case I, i.e.\ $0<|k|< \lfloor\frac{\sigma}{2} \rfloor$, is the most ``difficult case'' as the real part of the eigenvalues depends on $k$. However, one notices that the lower eigenvalue, $\lambda_{-,k}$, increases with $k$, which implies that, if there are $k-$s such that $0<|k|< \lfloor\frac{\sigma}{2} \rfloor$, the slowest possible convergence will be given by $\lambda_{-,\pm 1}$. As we need to compare the decay rates of \textit{all} modes \textit{simultaneously}, we find that it is enough to consider the following possibilities:
\begin{itemize}
\item $0<\sigma<2$: We only have possibilities of Case III, implying that all modes are controlled by $e^{-\frac{\sigma}{2}t}$.
\item $\sigma=2$: We have possibilities of Case III, as well as defectiveness in $k=\pm 1$ (Case II). This means that the modes are controlled by $\pa{1+t}e^{-t}$. If one searches for a \textit{pure exponential control}, the best rate one would find is $e^{-\pa{1-\epsilon}t}$ for any given fixed $\epsilon>0$.
\item $\sigma>2$: We have possibilities from Cases I and III, and potentially Case II. All the modes that correspond to Case I are controlled by $e^{-\pa{\frac{\sigma}{2} - \sqrt{\frac{\sigma^2}{4}-1}}t}$,
while those that correspond to Case III are controlled by $e^{-\frac{\sigma}{2}t}$. If Case II is realised, i.e. $\frac{\sigma}{2}\in \N\setminus \br{1}$, we find that the modes $k=\pm \frac{\sigma}{2}$ are controlled by $\pa{1+t}e^{-\frac{\sigma}{2}t}$. In total, thus, \textit{all} the modes are controlled by $e^{-\pa{\frac{\sigma}{2} - \sqrt{\frac{\sigma^2}{4}-1}}t}$, a decay rate that is realised on the $k=\pm 1$ modes, and the coefficient in the exponent is the spectral gap of the Goldstein-Taylor system \eqref{eq:gt_constant}.
\end{itemize} \begin{figure}[ht]
	\centering
\includegraphics[scale=0.65]{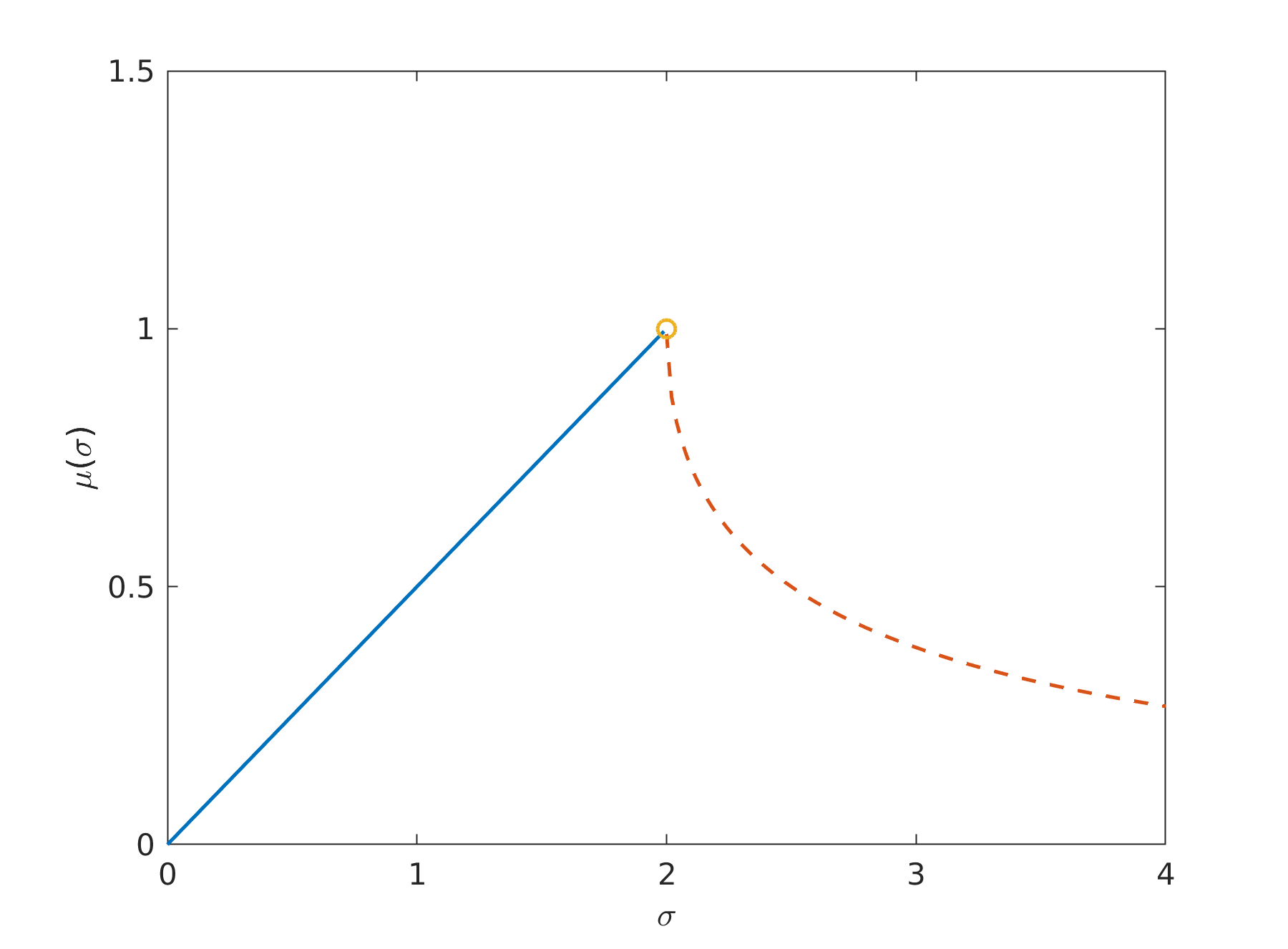}
	\caption{The exponential decay rate, $\mu\pa{\sigma}$, of the solution pair $(u(t)-u_{\mathrm{avg}},\,v(t))$ grows linearly until $\sigma=2$ where the defectiveness appears (hence the circle). From that point onwards the decay rate decreases, and is of order $O\pa{\frac{1}{\sigma}}$.}
	\label{fig:sigma}
\end{figure}
Before we turn our attention to properly consider these cases and ``uncover'' our spatial entropy, we remind the reader of the hypocoercivity technique which will allow us to transform the spectral information of $\C_k$ into a an appropriate, twisted norm with which we will show the desired decay of the $k$-th mode.

\subsection{Hypocoercivity and modal Lyapunov functionals}\label{sec4.2}
In the previous subsection we have concluded that, barring the zero mode, all the Fourier modes of \eqref{ODEsystem} decay exponentially (excluding potentially those with $|k|=\frac{\sigma}{2}$ where a polynomial correction is required). The lack of positive definiteness of the governing matrix, $\C_k$, stops us from seeing this behaviour in the Euclidean norm on $\CC^2$. However, by modifying the norm with the help of another, closely related,  positive definite matrix $\PP_k$, one can construct a new Lyaponov functional, which is equivalent to the Euclidean norm, that decays with the expected exponential rate (at least for a non-defective $\C_k$).\\
This is exactly the idea that motivated Arnold and Erb, and which is expressed in the following theorem (see \cite{Arnold2014}, \cite[Lemma 2]{AAC}):
\begin{theorem}\label{thm:anton_erb}
Let the matrix $\C\in\CC^{n\times n}$ be positive stable (i.e.\ have only eigenvalues with positive real parts). Let 
$$\mu=\min\br{\Re \lambda\;|\;\lambda \text{ is an eigenvalue of }\C}.$$
Then: 
\begin{enumerate}[i)]
\item\label{item:no_degen} If all eigenvalues with real part equal to $\mu$ are non-defective, there exists a Hermitian, positive definite matrix $\PP$ such that
\begin{equation}\label{eq:matrix_inequality}
\C^\ast \PP+\PP\C \geq 2\mu \PP.
\end{equation}
\item\label{item:degen} If at least one eigenvalue with real part equal to $\mu$ is defective, then for any $\epsilon>0$, one can find a Hermitian, positive definite matrix $\PP_{\epsilon}$ such that
\begin{equation}\label{eq:matrix_inequality_deg}
\C^\ast \PP_{\epsilon}+\PP_{\epsilon}\C \geq 2\pa{\mu-\epsilon} \PP_{\epsilon}\ ,
\end{equation}
where $\C^\ast$ denotes the Hermitian transpose of $\C$.
\end{enumerate}
\end{theorem}
We remark that the matrices $\PP$ and $\PP_\epsilon$ are never unique. \\
One can utilise the theorem in the following way: Assuming the eigenvalues associated to $\C$'s spectral gap, $\mu$, are non-defective, then by defining the norm 
$$\norm{y}^2_{\PP}:=\inner{y,\PP y}=y^\ast \PP y,$$ 
one sees that, if $y(t)$ solves the ODE $\dot{y}=-\C y$, then
\begin{equation}\label{Gronwall}
\frac{d}{dt}\norm{y}^2_{\PP}=-\inner{y, \pa{\C^\ast\PP+\PP \C} y} \leq -2 \mu \norm{y}^2_{\PP},
\end{equation}
resulting in the correct decay rate. The same approach works in the second case of Theorem \ref{thm:anton_erb}.\\
Besides the general idea of this methodology, Arnold and Erb have given a recipe (one that was later extended in 
\cite{AJW} to defective cases, using a time dependent matrix $\PP$) to finding the matrix $\PP,\PP_{\epsilon}$:\\
Assuming that $\C$ is diagonalisable, and letting $\br{\omega_i}_{i=1,\dots,n}$ be the eigenvectors of $\C^\ast$, the matrix  $\PP>0$ can be chosen to be
\begin{equation}\label{eq:form_of_P}
\PP=\sum_{i=1}^n b_i \omega_i \otimes \omega_i^\ast,
\end{equation}
for any positive sequence $\br{b_i}_{i=1,\dots, n}$. The above formula remains true, for \textit{a particular choice of} $\br{b_i}_{i=1,\dots, n}$, in the case where $\C$ is not diagonalisable. In that case we also need to augment the eigenvectors with the generalised eigenvectors. We refer the interested reader to Lemma 4.3 in \cite{Arnold2014}.
Moreover, for $n=2$, the case we shall need below, and $\C$ non-defective, all matrices  $\PP$ satisfying \eqref{eq:matrix_inequality} are indeed of the form \eqref{eq:form_of_P}, see \cite[Lemma 3.1]{AAS}.

We now turn our attention back to the Fourier transformed Goldstein-Taylor system \eqref{ODEsystem} and determine the modal Lyapunov functionals using the above recipe.
A short computation, where the weights $b_1,\,b_2$ are chosen such that both diagonal elements of $\PP$ are 1, finds the following matrices (For Case III we also require $b_1=b_2$, as this minimises the number of the resulting admissible matrices $\PP_k$ satisfying \eqref{eq:matrix_inequality}.):
\begin{itemize}
\item Case I: $0<|k|< \lfloor\frac{\sigma}{2} \rfloor$, with $\frac{\sigma}{2}\not\in \N$.  In this case we have:
\begin{equation}\label{eq:Pksmall}
\PP_k^{(I)}:= \begin{pmatrix}
1& -\frac{2ki}{\sigma}\\
\frac{2ki}{\sigma}& 1
\end{pmatrix},
\end{equation}
\item Case II: $|k|= \lfloor\frac{\sigma}{2}\rfloor\in \N$. As this case fosters defective eigenvalues, we will only consider the case $\sigma=2$ (as was mentioned beforehand), and state the matrix corresponding to $k=\pm 1$ and a given fixed $\epsilon>0$:
\begin{equation}\label{eq:Pkdeg}
\PP^{(II)}_{\epsilon,\pm 1}:= \begin{pmatrix}
1 & \mp\frac{i(2-\epsilon^2)}{2+\epsilon^2}\\
 \pm\frac{i(2-\epsilon^2)}{2+\epsilon^2} & 1
\end{pmatrix}
\end{equation}
\item Case III: $|k|> \lfloor\frac{\sigma}{2} \rfloor$. In this case we have:
\begin{equation}\label{eq:Pklarge}
\PP_k^{(III)}:= \begin{pmatrix}
1& -\frac{i\sigma }{2k}\\
\frac{i\sigma}{2k}& 1
\end{pmatrix}
\end{equation}
\end{itemize}
For each mode $k\ne0$, its {\it modal Lyapunov functional} will be given by $\big\|\binom{\hat u (k,t)}{\hat v (k,t)}\big\|_{\PP_k}^2$, where the matrix $\PP_k$ is chosen according to the above three cases. In Case II, the parameter $\epsilon>0$ can be chosen arbitrarily small.

\subsection{Derivation of the spatial entropy ${E_\theta(u,v)}$}\label{entropy-derivation}
The goal of this subsection is twofold: Finding a modal entropy to our system, and translating it to a spatial entropy that is modal-independent.\\
To begin with we shall define a {\it modal entropy} to quantify the exponential decay of solutions to \eqref{ODEsystem} towards its steady state:
\begin{equation}\label{F-steady}
  \widehat{u_\infty}(k) = \left\{
\begin{array}{ll}
\widehat{u_0}(k=0)=(u_0)_{\mathrm{avg}}, &  \: k=0  \\
0, &  \: k\ne0 \\
\end{array}
\right. \ ;\qquad\quad \widehat{v_\infty}(k)=0\ ,\: k\in\Z\ .
\end{equation}
Since the matrix $\C_0$ from \eqref{C0} has no spectral gap, the mode $k=0$ plays a special role, and hence will be treated separately. \\
Once found, we will want to relate that modal-based entropy to the {\it spatial entropy} $E_\theta$ from Definition \ref{def:entropy}, which is not based on a modal decomposition. To this end we already remark that the off-diagonal factors $ik$ in \eqref{eq:Pksmall} and $1/ik$ in \eqref{eq:Pklarge} correspond in physical space, roughly speaking, to a first derivative and  an anti-derivative, respectively. \\

As in \S\ref{subsec:Fourier} we shall distinguish three cases of $\sigma$:\\
$\bm{0<\sigma<2}:$ All modes $k\ne0$ satisfy $|k|> \lfloor\frac{\sigma}{2} \rfloor$, and hence are of Case III. We recall from \S\ref{subsec:Fourier} that all modes decay here with the sharp rate $\frac{\sigma}{2}$. For a modal entropy to reflect this decay, we hence have to use for each mode a Lyapunov functional $\big\|\binom{\hat u (k,t)}{\hat v (k,t)}\big\|_{\PP_k}^2$, where $\PP_k$ satisfies the inequality \eqref{eq:matrix_inequality} with $\mu=\frac{\sigma}{2}$. $\PP_k=\PP_k^{(III)}$ is the most convenient choice. \\
We define the modal entropy for any $\br{\widehat{u}(k),\widehat{v}(k)}_{k\in\Z}$ such that $\widehat{u}(0)=0$ as
\begin{eqnarray}\label{modalE-sigma<2}
\mathcal{E}\pa{\widehat{u},\widehat{v}}&:=&\sum_{k\in\Z\setminus \br{0}}\norm{\begin{pmatrix}
\widehat{u}(k)\\ \widehat{v}(k)
\end{pmatrix}}^2_{\PP_k^{(III)}}+\norm{\begin{pmatrix}
\widehat{u}(0)\\ \widehat{v}(0)
\end{pmatrix}}^2\\
&=&\sum_{k\in\Z}\pa{\abs{\widehat{u}(k)}^2  -\sigma\Re\pa{\frac{\widehat{u}(k)}{ik}\overline{\widehat{v}(k)}}+\abs{\widehat{v}(k)}^2}\ ,
\end{eqnarray}
where we used the convention $\frac{\widehat{u}(0)}{0}=0$.
The mode $k=0$ was included since $\widehat{u}(0,t)=\widehat{u}(0)=0$ and $\widehat{v}(0,t)=\widehat{v}(0)e^{-\sigma t}$. 
Using Plancherel's equality, and \eqref{item:cont_and_fourier_coefficients} from Lemma \ref{lem:properties_of_antiderivative}, we find that 
\begin{equation}\label{EE=E}
\mathcal{E}\pa{\widehat{u},\widehat{v}}=E_\sigma\pa{u,v},
\end{equation}
which shows why we consider the spatial entropy functional from Definition \ref{def:entropy} in this case.\\
We note that, since $u_{\mathrm{avg}}(t)$ is conserved, part \eqref{item:cont_and_fourier_coefficients} of Lemma \ref{lem:properties_of_antiderivative}, explains why we have chosen to use the anti-derivative of $u$, and not of $v$.\\

\noindent
$\bm{\sigma>2:}$ This situation is more complicated than the previous one, as we have a mixture of at least two of the aforementioned three cases: finitely many $k-$s in $\Z$ for which $0<|k|< \lfloor\frac{\sigma}{2} \rfloor$ (i.e.\ Case I), Case II for two $k-$s if $\frac{\sigma}{2}\in\N$, while the rest satisfy $|k|> \lfloor\frac{\sigma}{2} \rfloor$ (i.e.\ Case III).
Following the above methodology to construct the modal entropy, we would need to use a combination of $\PP_k^{(I)}$ and $\PP_k^{(III)}$, given by \eqref{eq:Pksmall} and \eqref{eq:Pklarge}, and potentially a matrix for the defective modes. This is feasible on the modal level, but does not easily translate back to the spatial variables. It would yield a complicated pseudo-differential operator ``inside'' the spatial entropy. \\
Recalling the discussion from \S\ref{subsec:Fourier} we see that the overall decay rate, $\mu=\frac{\sigma}{2}-\sqrt{\frac{\sigma^2}{4}-1}$ is only determined by the modes $k=\pm1$. Since all the other modes decay faster, we are not obliged to use ``optimal'' modal Lyapunov functionals for these higher modes. This gives some leeway for choosing the matrices $\PP_k$, $|k|>1$.

For $k\ne0$ we shall use in fact the matrix 
\begin{equation}\label{Psuff}
\PP_k^{\mathrm{suff}}:=\PP_{k}^{(III)}\pa{\sigma \to \frac{4}{\sigma}}=\begin{pmatrix}
1& -\frac{2i}{k\sigma}\\
\frac{2i}{k\sigma}& 1
\end{pmatrix} >0\ ,
\end{equation}
which satisfies $\PP_{\pm1}^{\mathrm{suff}}=\PP_{\pm1}^{(I)}$ for the crucial lowest modes. 
It also satisfies the following result, which implies exponential decay of all modal Lyapunov functionals $\big\|\binom{\hat u (k,t)}{\hat v (k,t)}\big\|_{\PP_k^{\mathrm{suff}}}^2$, $k\ne0$ with rate $2\mu=\sigma-\sqrt{\sigma^2-4}$.
\begin{lemma}\label{P_sigma>2}
Let $\sigma>2$. Then
$$
  \C^\ast_k \PP^{\mathrm{suff}}_k+\PP^{\mathrm{suff}}_k\C_k-2\mu \PP^{\mathrm{suff}}_k\ge0
  \qquad \forall\,k\ne0\ .
$$
\end{lemma}
The proof of this lemma is straightforward\footnote{In a sense, the same computation that shows this inequality is embedded in the proof of the exponential decay of $E_\theta$ in the next subsection.}.
Proceeding like in \eqref{modalE-sigma<2} we define the modal entropy for any $\br{\widehat{u}(k),\widehat{v}(k)}_{k\in\Z}$ such that $\widehat{u}(0)=0$:
$$\mathcal{E}\pa{\widehat{u},\widehat{v}}:=\sum_{k\in\Z\setminus \br{0}}\norm{\begin{pmatrix}
\widehat{u}(k)\\ \widehat{v}(k)
\end{pmatrix}}^2_{\PP_k^{\mathrm{suff}}}+\norm{\begin{pmatrix}
\widehat{u}(0)\\ \widehat{v}(0)
\end{pmatrix}}^2\ .$$
Due to \eqref{EE=E} and \eqref{Psuff} it is related to the spatial entropy functional from Definition \ref{def:entropy} as
$$\mathcal{E}\pa{\widehat{u},\widehat{v}}=E_{\frac{4}{\sigma}}\pa{u,v}.$$\\

\noindent
$\bm{\sigma=2:}$ Just like in the previous case, the lowest frequency modes $k=\pm 1$ control the large time behaviour. However, the matrices $\C_{\pm1}$ are now defective, which leads to a (purely) exponential decay rate reduced by $\epsilon$.\\
We proceed similarly to the case $\sigma>2$ and define for some $\epsilon>0$:
\begin{equation}\label{Psuff1}
\PP_{\epsilon,k}^{\mathrm{suff}}=
\PP^{(III)}_{k}\pa{\sigma \to \frac{2\pa{2-\epsilon^2}}{2+\epsilon^2}}
=\begin{pmatrix}
1& -\frac{i\pa{2-\epsilon^2}}{k\pa{2+\epsilon^2}}\\
\frac{i\pa{2-\epsilon^2}}{k\pa{2+\epsilon^2}}& 1
\end{pmatrix}>0\ ,
\end{equation}
which satisfies $\PP_{\epsilon,\pm1}^{\mathrm{suff}}=\PP^{(II)}_{\epsilon,\pm 1}$ for the crucial lowest model.  
It also satisfies the following result, which implies exponential decay of all modal Lyapunov functionals $\big\|\binom{\hat u (k,t)}{\hat v (k,t)}\big\|_{\PP_{\epsilon,k}^{\mathrm{suff}}}^2$, $k\ne0$ with rate of at least $2\mu=2(1-\epsilon)$.
\begin{lemma}\label{P_sigma=2}
Let $\sigma=2$. Then
$$
  \C^\ast_k \PP^{\mathrm{suff}}_{\epsilon,k}+\PP^{\mathrm{suff}}_{\epsilon,k}\C_k-2\mu \PP^{\mathrm{suff}}_{\epsilon,k}>0
  \qquad \forall\,k\ne0\ .
$$
\end{lemma}
Proceeding like in \eqref{modalE-sigma<2} we define the modal entropy for any $\br{\widehat{u}(k),\widehat{v}(k)}_{k\in\Z}$ such that $\widehat{u}(0)=0$:
$$\mathcal{E}\pa{\widehat{u},\widehat{v}}:=\sum_{k\in\Z\setminus \br{0}}\norm{\begin{pmatrix}
\widehat{u}(k)\\ \widehat{v}(k)
\end{pmatrix}}^2_{\PP_{\epsilon,k}^{\mathrm{suff}}}+\norm{\begin{pmatrix}
\widehat{u}(0)\\ \widehat{v}(0)
\end{pmatrix}}^2\ .$$
Due to \eqref{EE=E} and \eqref{Psuff1} it is related to the spatial entropy functional from Definition \ref{def:entropy} as
$$\mathcal{E}\pa{\widehat{u},\widehat{v}}=E_{\frac{2\pa{2-\epsilon^2}}{2+\epsilon^2}}\pa{u,v}.$$\\

\subsection{The evolution of the spatial entropy $E_\theta$}\label{subsec:entropy-evol}
In the previous subsection we have shown how, depending on the value of $\sigma$, the entropies $E_{\sigma}$, $E_{\frac{4}{\sigma}}$ and $E_{\frac{2\pa{2-\epsilon^2}}{2+\epsilon^2}}$ are the correct candidates to show the exponential convergence to equilibrium. A closer look at \eqref{Gronwall} shows that each modal Lyapunov functional $\big\|\binom{\hat u (k,t)}{\hat v (k,t)}\big\|_{\PP_k}^2$ decays exponentially, and hence also the spatial entropy $E_\theta$. Recalling the decay rates presented in \S\ref{entropy-derivation} for the three regimes of $\sigma$, confirms that we have actually already proved most of part \eqref{item:convergence_sigma_constant} of Theorem \ref{thm:main}. However, as our main goal is to consider these functionals in the spatial variable alone (i.e.\ without a modal decomposition), we shall show how one achieves the correct convergence result following a direct calculation. This will also serve as a preparation for \S\ref{sec:sigma_general}.
\begin{theorem}\label{thm:diff_ineq_sigma_constant}
Under the same conditions of Theorem \ref{thm:main} with $\sigma(x)=\sigma$, one has that
\begin{enumerate}[i)]
\item If $0<\sigma<2$ then
$$E_{\sigma}\pa{u(t)-u_{\mathrm{avg}},v(t)} \leq E_{\sigma}\pa{u_{0}-u_{\mathrm{avg}},v_0}e^{- \sigma t}.$$
\item If $\sigma>2$ then 
$$E_{\frac{4}{\sigma}}\pa{u(t)-u_{\mathrm{avg}},v(t)} \leq E_{\frac{4}{\sigma}}\pa{u_{0}-u_{\mathrm{avg}},v_0}e^{- \pa{\sigma-\sqrt{\sigma^2-4}} t}.$$
\item If $\sigma=2$ then for any $0<\epsilon<1$
$$E_{\frac{2\pa{2-\epsilon^2}}{2+\epsilon^2}}\pa{u(t)-u_{\mathrm{avg}},v(t)} \leq E_{\frac{2\pa{2-\epsilon^2}}{2+\epsilon^2}}\pa{u_{0}-u_{\mathrm{avg}},v_0}e^{- 2(1-\epsilon) t}.$$
\end{enumerate}
\end{theorem}
\begin{proof}
In order to prove this theorem we shall obtain differential inequalities for $E_\theta$, from which we will conclude the desired result by a simple application of Gronwall's inequality. Using Proposition \ref{thm:evolution_of_entropy} we find that:\\
\underline{If $0<\sigma<2$:}
\begin{equation}\nonumber
\begin{gathered}
\frac{d}{dt}E_{\sigma}\pa{u(t)-u_{\mathrm{avg}},v(t)}=-\sigma \norm{u(t)-u_{\mathrm{avg}}}^2 - \sigma\norm{v(t)}^2\\
+\frac{\sigma^2}{2\pi}\int_0^{2\pi}  \partial_x^{-1} \pa{u(x,t)-u_{\mathrm{avg}}}v(x,t) dx-\sigma\pa{v(t)_{\mathrm{avg}}}^2
\end{gathered}
\end{equation}
\begin{equation}\nonumber
=-\sigma E_{\sigma}\pa{u(t)-u_{\mathrm{avg}},v(t)} -\sigma \pa{v(t)_{\mathrm{avg}}}^2 \leq -\sigma E_{\sigma}\pa{u(t)-u_{\mathrm{avg}},v(t)}.
\end{equation}
Note that, since $v_{\mathrm{avg}}(t)=(v_{0})_{\mathrm{avg}}\,e^{-\sigma t}$, 
we can compute $E_{\theta}\pa{u(t)-u_{\mathrm{avg}},v(t)}$ explicitly.\\
\underline{If $\sigma>2$:} 
\begin{equation}\nonumber
\begin{gathered}
\frac{d}{dt}E_{\frac{4}{\sigma}}\pa{u(t)-u_{\mathrm{avg}},v(t)}=-\frac{4}{\sigma} \norm{u(t)-u_{\mathrm{avg}}}^2 -\pa{2\sigma-\frac{4}{\sigma}}\norm{v(t)}^2\\
+\frac{4}{2\pi}\int_0^{2\pi}  \partial_x^{-1} \pa{u(x,t)-u_{\mathrm{avg}}} v(x,t)dx-\frac{4}{ \sigma} \pa{v(t)_{\mathrm{avg}}}^2\\
\end{gathered}
\end{equation}
\begin{equation}\nonumber
\begin{gathered}
\leq -\pa{\sigma-\sqrt{\sigma^2-4}} E_{\frac{4}{\sigma}}\pa{u(t)-u_{\mathrm{avg}},v(t)} +\pa{\sigma-\sqrt{\sigma^2-4}-\frac{4}{\sigma}}\norm{u(t)-u_{\mathrm{avg}}}^2\\
+\pa{\frac{4}{\sigma}-\sigma-\sqrt{\sigma^2-4}}\norm{v(t)}^2+\frac{4}{2\pi}\pa{1-\frac{\sigma-\sqrt{\sigma^2-4}}{\sigma}}\int_0^{2\pi} \partial_x^{-1}  \pa{u(x,t)-u_{\mathrm{avg}}}v(x,t) dx.
\end{gathered}
\end{equation}
The desired inequality,  
{$\frac{d}{dt}E_{\frac{4}{\sigma}}\le -\big(\sigma-\sqrt{\sigma^2-4}\,\big) E_{\frac{4}{\sigma}}$}, is valid if and only if
\begin{equation}\label{eq:four_over_sigma_proof}
\begin{gathered}
\frac{4}{2\pi}\int_0^{2\pi} \partial_x^{-1}  \pa{u(x,t)-u_{\mathrm{avg}}} v(x,t)dx \\
\leq \pa{\sigma-\sqrt{\sigma^2-4}}\norm{u(t)-u_{\mathrm{avg}}}^2+\pa{\sigma+\sqrt{\sigma^2-4}}\norm{v(t)}^2.
\end{gathered}
\end{equation}
Cauchy-Schwarz inequality, together with Poincar\'e inequality (Lemma \ref{thm:poincare}) and Lemma \ref{lem:properties_of_antiderivative}, imply that
$$\frac{4}{2\pi}\int_0^{2\pi} \partial_x^{-1}  \pa{u(x,t)-u_{\mathrm{avg}}}v(x,t) dx \leq 4\norm{u(t)-u_{\mathrm{avg}}}\norm{v(t)}$$
$$=2\pa{\sqrt{\sigma-\sqrt{\sigma^2-4}}\norm{u(t)-u_{\mathrm{avg}}}}\pa{\sqrt{\sigma+\sqrt{\sigma^2-4}}\norm{v(t)}}\ .$$
Together with the fact that $2\abs{ab}\leq a^2+b^2$ this shows \eqref{eq:four_over_sigma_proof}, concluding the proof in this case.\\
\underline{If $\sigma=2$:} 
\begin{equation}\nonumber
\begin{gathered}
\frac{d}{dt}E_{\frac{2\pa{2-\epsilon^2}}{2+\epsilon^2}}\pa{u(t)-u_{\mathrm{avg}},v(t)}=- \frac{2\pa{2-\epsilon^2}}{2+\epsilon^2}\norm{u(t)-u_{\mathrm{avg}}}^2 -\frac{2\pa{2+3\epsilon^2}}{2+\epsilon^2}\norm{v(t)}^2\\
+\frac{1}{2\pi}\cdot \frac{4\pa{2-\epsilon^2}}{2+\epsilon^2}\int_0^{2\pi}  \partial_x^{-1} \pa{u(x,t)-u_{\mathrm{avg}}}v(x,t) dx-\frac{2\pa{2-\epsilon^2}}{2+\epsilon^2} \pa{v(t)_{\mathrm{avg}}}^2\\
\end{gathered}
\end{equation}
\begin{equation}\nonumber
\begin{gathered}
\leq -2\pa{1-\epsilon}E_{\frac{2\pa{2-\epsilon^2}}{2+\epsilon^2}}\pa{u(t)-u_{\mathrm{avg}},v(t)} -2\epsilon\pa{1-\frac{2\epsilon}{2+\epsilon^2}}\norm{u(t)-u_{\mathrm{avg}}}^2\\
-2\epsilon\pa{1+\frac{2\epsilon}{2+\epsilon^2}}\norm{v(t)}^2+\frac{1}{2\pi}\cdot \frac{4\epsilon\pa{2-\epsilon^2}}{2+\epsilon^2}\int_0^{2\pi}  \partial_x^{-1} \pa{u(x,t)-u_{\mathrm{avg}}}v(x,t) dx.
\end{gathered}
\end{equation}
Like before, the desired inequality will follow if 
$$\frac{1}{2\pi}\cdot \frac{2\pa{2-\epsilon^2}}{2+\epsilon^2} \int_0^{2\pi}  \partial_x^{-1} \pa{u(x,t)-u_{\mathrm{avg}}}v(x,t)dx$$
$$\leq \pa{1-\frac{2\epsilon}{2+\epsilon^2}}\norm{u(t)-u_{\mathrm{avg}}}^2+\pa{1+\frac{2\epsilon}{2+\epsilon^2}}\norm{v(t)}^2.$$
This is valid since
$$\frac{1}{2\pi}\cdot \frac{2\pa{2-\epsilon^2}}{2+\epsilon^2} \int_0^{2\pi}\partial_x^{-1} \pa{u(x,t)-u_{\mathrm{avg}}}v(x,t)dx$$
$$\leq \frac{2\sqrt{4+\epsilon^4}}{2+\epsilon^2}\norm{u(t)-u_{\mathrm{avg}}}\norm{v(t)}= 2 \pa{\sqrt{1-\frac{2\epsilon}{2+\epsilon^2}}\norm{u(t)-u_{\mathrm{avg}}}}\pa{\sqrt{1+\frac{2\epsilon}{2+\epsilon^2}}\norm{v(t)}}$$
$$\leq \pa{1-\frac{2\epsilon}{2+\epsilon^2}}\norm{u(t)-u_{\mathrm{avg}}}^2+\pa{1+\frac{2\epsilon}{2+\epsilon^2}}\norm{v(t)}^2,$$
where we used Cauchy-Schwarz inequality, Poincar\'e inequality, and Lemma \ref{lem:properties_of_antiderivative} again.\\
The theorem is now complete.
\end{proof}
As the last part of this section, we finally prove part \eqref{item:convergence_sigma_constant} of Theorem \ref{thm:main}:
\begin{proof}[Proof of part \eqref{item:convergence_sigma_constant} of Theorem \ref{thm:main}]
The decay estimates of $E_{\theta(\sigma)}$ are already shown in Theorem \ref{thm:diff_ineq_sigma_constant}. To show \eqref{eq:main_for_f_pm_constant} and \eqref{eq:main_for_f_pm_constant_ep} we recall that
$$f_+ = \frac{u+v}{2},\quad f_-=\frac{u-v}{2}\ ,$$ 
and
$$\norm{f}^2+\norm{g}^2\leq \frac{2}{2-\theta}E_{\theta}\pa{f,g},\quad E_{\theta}\pa{f,g}\leq \frac{2+\theta}{2}\pa{\norm{f}^2+\norm{g}^2}$$
for $0<\theta<2$ and $f_{\mathrm{avg}}=0$, according to Lemma \ref{lem:entropy_equivalence}. Thus, using the definition of $f_\infty$ from \eqref{eq:C_sigma_and_f_infty} we see that
$$\norm{f_+(t)-f_\infty}^2 + \norm{f_-(t)-f_\infty}^2  
$$
$$ =\frac12\norm{u(t)-u_{\mathrm{avg}}}^2+\frac12\norm{v(t)}^2 \leq \frac{1}{2-\theta}E_\theta \pa{u(t)-u_{\mathrm{avg}},v(t)}$$
$$ \leq \frac{1}{2-\theta}E_\theta \pa{u_0-u_{\mathrm{avg}},v_0}e^{-2\mu(\sigma)t} \leq \frac12\cdot\frac{2+\theta}{2-\theta}\pa{\norm{u_0-u_{\mathrm{avg}}}^2+\norm{v_0}^2}e^{-2\mu(\sigma)t}$$
$$=\frac{2+\theta}{2-\theta} \pa{\norm{f_{+,0}-f_\infty}^2+\norm{f_{-,0}-f_\infty}^2} e^{-2\mu(\sigma)t},$$
which shows the result for the appropriate choices of $\theta(\sigma)$ and $\mu(\sigma)$. For $\sigma=2$ we choose
$$\theta(2)=\frac{2\pa{2-\epsilon^2}}{2+\epsilon^2},\quad \mu(2)=1-\epsilon\ .$$

The sharpness of the decay rate for $\sigma\ne2$ can be verified easily on the first mode, e.g.\ for $u_0=0$, $v_0=e^{ix}$.
\end{proof}
With the constant case fully behind us, we can now focus on the case where $\sigma(x)$ is a non-constant function.


\section{$x-$dependent relaxation function}\label{sec:sigma_general}
The large time behaviour of solutions to the Goldstein-Taylor equation \eqref{eq:gt}, or equivalently its recast form \eqref{eq:gt_recast}, becomes increasingly harder to understand, if the relaxation function, $\sigma(x)$, is not a constant. However, as shown in \S\ref{sec:sigma_constant}, we have managed to find a potential spatial entropy that captures the exact behaviour of the decay to equilibrium. The idea that we will employ in this section is to use the same type of entropy to try and estimate the convergence rate \textit{even when $\sigma(x)$ is not constant.} This is, as mentioned in the introduction, a perturbative approach - yet the methodology, and ideas, are robust enough to deal with more complicated systems, as will be shown in the next section.

A fundamental theorem to establish our main result, Theorem \ref{thm:main} \eqref{item:convergence_sigma_not_constant}, is the following:
\begin{theorem}\label{thm:perturbation}
Let $u,v\in C([0,\infty);L^2\pa{\T})$ be mild solutions to \eqref{eq:gt_recast} with initial datum $u_0,\; v_0\in L^2\pa{\T}$. Denoting by $u_{\mathrm{avg}}=\pa{u_0}_{\mathrm{avg}}$ we have that for any given $0<\alpha,\theta<2 $ the conditions
\begin{equation}\label{eq:perturbative_condition_I}
\begin{gathered}
\alpha < \theta,\quad  \theta+\alpha < 2\smin
\end{gathered}
\end{equation}
and
\begin{equation}\label{eq:perturbative_condition_II}
\begin{gathered}
\sup_{x\in\T}\pa{\theta^2\pa{\sigma(x)-\alpha}^2-4\pa{\theta-\alpha}\pa{2\sigma(x)-\theta-\alpha}} \leq 0,
\end{gathered}
\end{equation}
imply that 
\begin{equation}\label{eq:perturbation}
E_{\theta}\pa{u(t)-u_{\mathrm{avg}},v(t)} \leq E_{\theta}\pa{u_0-u_{\mathrm{avg}},v_0}e^{-\alpha t},\quad t\ge0. 
\end{equation}
\end{theorem}  
\begin{proof}
Using \eqref{eq::evolution_of_entropy} from Proposition \ref{thm:evolution_of_entropy}, and the fact that $\theta \pa{v(t)_{\mathrm{avg}}}^2 \geq 0$, we find that
\begin{equation}\label{E-diff-inequ}
\begin{gathered}
\frac{d}{dt}E_{\theta}\pa{u(t)-u_{\mathrm{avg}},v(t)}\leq -\alpha E_{\theta}\pa{u(t)-u_{\mathrm{avg}},v(t)}-\pa{\theta-\alpha} \norm{u(t)-u_{\mathrm{avg}}}^2 \\
- \frac{1}{2\pi} \int_0^{2\pi}( 2\sigma(x)-\theta-\alpha)v(x,t)^2 dx
+\frac{\theta}{2\pi}\int_0^{2\pi}\pa{\sigma(x)-\alpha}\partial_x^{-1} \pa{u(x,t)-u_{\mathrm{avg}}} v(x,t)  dx.
\end{gathered}
\end{equation}
The proof of the theorem will follow from the above inequality if we can show that 
\begin{equation}\label{eq:perturbation_proof}
\begin{gathered}
\frac{\theta}{2\pi}\int_0^{2\pi}\pa{\sigma(x)-\alpha}\partial_x^{-1} \pa{u(x,t)-u_{\mathrm{avg}}} v(x,t)  dx \\
\leq \pa{\theta-\alpha} \norm{u(t)-u_{\mathrm{avg}}}^2 +\frac{1}{2\pi} \int_0^{2\pi}( 2\sigma(x)-\theta-\alpha)v(x,t)^2 dx.
\end{gathered}
\end{equation}
Due to condition \eqref{eq:perturbative_condition_I} we have that
$$\inf_{x\in \T}\pa{2\sigma(x)-\theta-\alpha}=2\smin-\theta-\alpha>0.$$
Hence, we obtain with Cauchy-Schwarz, Young's inequality $\abs{ab} \leq \frac{a^2}{\theta}+\frac{\theta b^2}{4}$, and the Poincar\'e inequality, \eqref{eq:poincare}, that
\begin{equation}\label{eq:perturbation_for_app}
\begin{gathered}
\abs{\frac{\theta}{2\pi}\int_0^{2\pi}\pa{\sigma(x)-\alpha}  \partial_x^{-1} \pa{u(x,t)-u_{\mathrm{avg}}} v(x,t)  dx }\\
\leq \frac{\theta}{2\pi}\int_0^{2\pi}\sqrt{2\sigma(x)-\theta-\alpha} \abs{v(x,t)} \frac{\abs{\sigma(x)-\alpha}}{\sqrt{2\sigma(x)-\theta-\alpha}}\abs{\partial_x^{-1} \pa{u(x,t)-u_{\mathrm{avg}}} }dx 
\end{gathered}
\end{equation}
\begin{equation}\nonumber
\begin{gathered}
\leq \frac{\theta}{2\pi}\pa{\int_0^{2\pi}\pa{2\sigma(x)-\theta-\alpha} v(x,t)^2dx}^{\frac{1}{2}}\pa{\int_{0}^{2\pi}\frac{\pa{\sigma(x)-\alpha}^2}{2\sigma(x)-\theta-\alpha} \pa{\partial_x^{-1} \pa{u(x,t)-u_{\mathrm{avg}}} }^2dx }^{\frac{1}{2}} \\
\end{gathered}
\end{equation}
\begin{equation}\nonumber
\begin{gathered}
\leq \frac{1}{2\pi} \int_0^{2\pi}( 2\sigma(x)-\theta-\alpha)v(x,t)^2 dx +
\frac{1}{2\pi} \int_0^{2\pi}\frac{\theta^2\pa{\sigma(x)-\alpha}^2} {4\pa{2\sigma(x)-\theta-\alpha}}\pa{\ad\pa{u(x,t)-u_{\mathrm{avg}}}}^2 dx \\
\leq \frac{1}{2\pi} \int_0^{2\pi}( 2\sigma(x)-\theta-\alpha)v(x,t)^2 dx + 
\sup_{x\in\T}\pa{\frac{\theta^2\pa{\sigma(x)-\alpha}^2}{4\pa{2\sigma(x)-\theta-\alpha}}}\norm{u(t)-u_{\mathrm{avg}}}^2 .
\end{gathered}
\end{equation}
The above implies that \eqref{eq:perturbation_proof} will be valid when
$$\sup_{x\in\T}\frac{\theta^2\pa{\sigma(x)-\alpha}^2} {4\pa{2\sigma(x)-\theta-\alpha}} \leq \theta-\alpha,$$
which is equivalent, due to the positivity of the denominator, to \eqref{eq:perturbative_condition_II}. The proof is thus complete.
\end{proof}
\begin{remark}\label{rem:necessity_of_conditions}
It is worth to note that the conditions expressed in \eqref{eq:perturbative_condition_I} are crucial in our estimation. Indeed, they tell us that  
$$\pa{\theta-\alpha} \norm{u(t)-u_{\mathrm{avg}}}^2\quad\text{and}\quad  \int_0^{2\pi}( 2\sigma(x)-\theta-\alpha)v(x,t)^2 dx$$
are non-negative. If one part of the condition would not be true, we would be able to ``cook'' initial data such that the mixed $u$--$v$--term {in \eqref{eq:perturbation_proof}} is zero, and the above terms add up to something strictly negative - breaking the functional inequality we are aiming to attain.
\end{remark}
The next step towards proving part \eqref{item:convergence_sigma_not_constant} in Theorem \ref{thm:main} is to look for $\theta$ and $\alpha$ such that conditions \eqref{eq:perturbative_condition_I} and \eqref{eq:perturbative_condition_II} are satisfied. \\
We recall the definition of $\theta^\ast$ from Theorem \ref{thm:main}:
\begin{equation}\nonumber
\theta^\ast:=\min\pa{\smin,\frac{4}{\smax}},
\end{equation}
which in a sense captures the ``worst possible'' behaviour when comparing $\sigma(x)$ to the constant case (with $\sigma\not=2$). We show the following:
\begin{lemma}\label{lem:value_of_alpha}
Assume that $0<\smin<\smax<\infty$, where $\smin$ and $\smax$ were defined in Theorem \ref{thm:main}. Then 
\begin{equation}\nonumber
\alpha^\ast:=\alpha^\ast\pa{\smin,\smax}:=\begin{cases}
\frac{\smin\pa{4+2\sqrt{4-\smin^2}-\smin\smax}}{4+2\sqrt{4-\smin^2}-\smin^2},  & \smin < \frac{4}{\smax}\\
\smax - \sqrt{\smax^2-4}, & \smin \geq \frac{4}{\smax} 
\end{cases}
\end{equation}
is such that $\theta^\ast$ and $\alpha^\ast$ satisfy conditions \eqref{eq:perturbative_condition_I} and \eqref{eq:perturbative_condition_II}.
\end{lemma} 
\begin{proof}
Clearly, since
$$\theta^\ast \leq \begin{cases} \smin, & \smin <\smax \leq 2 \\  \frac{4}{\smax}, & \smax >2  \end{cases}$$
we always have that $0<\theta^\ast <2$. \\
We continue by considering condition \eqref{eq:perturbative_condition_II}, and finding appropriate parameters which will give condition \eqref{eq:perturbative_condition_I} automatically. Denoting by
$$f\pa{\alpha,\theta,y}:=\theta^2\pa{y-\alpha}^2-4\pa{\theta-\alpha}\pa{2y-\theta-\alpha}$$
for $\pa{\alpha,\theta}$ that satisfy condition \eqref{eq:perturbative_condition_I} and $y\in\rpa{\smin,\smax}$, we find that for fixed $\alpha$ and $\theta$, $f$ is an upward parabola in $y$ whose non-positive part lies between its roots
$$y_{\pm}\pa{\alpha,\theta}:=\alpha + \frac{2\pa{\theta-\alpha}}{\theta^2}\pa{2\pm \sqrt{4-\theta^2}}.$$
Thus, condition \eqref{eq:perturbative_condition_II} is satisfied if and only if
$$y_{-}\pa{\alpha,\theta}\leq \smin,\quad\text{and}\quad \smax \leq y_{+}\pa{\alpha,\theta}.$$
A simple calculation shows that for $0<\theta<2$
$$y_{-}\pa{\alpha,\theta}\leq \smin\;\;\Leftrightarrow\;\; \alpha \leq \frac{\theta\pa{2\sqrt{4-\theta^2}-\pa{4-\smin \theta}}}{2\sqrt{4-\theta^2}-\pa{4-\theta^2}}=:\gamma_{\mathrm{min}}\pa{\theta},$$
$$\smax \leq y_{+}\pa{\alpha,\theta}\;\;\Leftrightarrow\;\; \alpha \leq \frac{\theta\pa{2\sqrt{4-\theta^2}+\pa{4-\smax \theta}}}{2\sqrt{4-\theta^2}+\pa{4-\theta^2}}=:\gamma_{\mathrm{max}}\pa{\theta}.$$
This means that, if we choose $\alpha\pa{\theta}$ for a fixed $\theta$ so that condition \eqref{eq:perturbative_condition_II} is valid, we must have that
$$\alpha\pa{\theta} \leq \min\pa{\gamma_{\mathrm{min}}\pa{\theta},\gamma_{\mathrm{max}}\pa{\theta}}.$$
One can continue and show that (see Appendix \ref{appsec:proofs}):
\begin{enumerate}[(i)]
\item\label{item:max_less_min_gamma} For $\theta \leq \smin$ and $0<\theta<2$ we have that $\gamma_{\mathrm{max}}\pa{\theta} \leq \gamma_{\mathrm{min}}\pa{\theta}$.
\item\label{item:max_less_theta} For $\theta \leq \frac{4}{\smax}$ and $0<\theta<\smax$ we have that $0<\gamma_{\mathrm{max}}\pa{\theta}< \theta$.
\end{enumerate}
With these observations we deduce that for any 
$$
\theta\in (0,\theta^\ast]= \left(0,\min\pa{\smin,\frac{4}{\smax}} \right]\cap (0,2)
$$
we have $\theta<\smax$ and hence
$$\gamma_{\mathrm{max}}(\theta)=\min\pa{\gamma_{\mathrm{min}}(\theta),\gamma_{\mathrm{max}}(\theta)}\quad \text{and} \quad\gamma_{\max}\pa{\theta}<\theta.$$
Hence, the pair $\pa{\theta,\alpha=\gamma_{\mathrm{max}}(\theta)}$ satisfies not only condition \eqref{eq:perturbative_condition_II} but also
$$\gamma_{\mathrm{max}}(\theta)+\theta < 2\theta \leq 2\theta^\ast \leq 2\smin\quad \text{and}\quad \gamma_{\mathrm{max}}\pa{\theta}< \theta,$$
i.e. condition \eqref{eq:perturbative_condition_I}.
We conclude that $\theta$ and $\alpha=\gamma_{\mathrm{max}}\pa{\theta}$ satisfy both desired conditions, for any $\theta\in (0,\theta^\ast]$.\\
Noticing that 
$$\gamma_{\mathrm{max}}\pa{\theta^\ast}=
\left\{\begin{aligned}
& \frac{\smin\pa{2\sqrt{4-\smin^2}+\pa{4-\smax \smin}}}{2\sqrt{4-\smin^2}+\pa{4-\smin^2}}, & \smin < \frac{4}{\smax} \\
& \frac{\frac{8}{\smax}\sqrt{4-\frac{16}{\smax^2}}}{2\sqrt{4-\frac{16}{\smax^2}}+ 4-\frac{16}{\smax^2}}, &\smin>\frac{4}{\smax}
\end{aligned}\right\}
=\alpha^\ast\pa{\smin,\smax},$$
we conclude the proof.
\end{proof}
\begin{remark}\label{rem:optimisation_of_gamma}
The choice of $\alpha^\ast\pa{\smin,\smax}=\gamma_{\mathrm{max}}\pa{\theta^\ast}$ is not accidental. Indeed, one can easily show that 
$$\frac{d}{d\theta}\gamma_{\mathrm{max}}\pa{\theta}=\frac{8-2\smax \theta}{\pa{4-\theta^2}^{\frac{3}{2}}},$$
and as such 
$$\max_{\theta\in(0,\theta^\ast]}\gamma_{\mathrm{max}}\pa{\theta}=\gamma_{\mathrm{max}}\pa{\theta^\ast}.$$
As the parameter $\alpha^\ast=\gamma_{\mathrm{max}}\pa{\theta^\ast}$ corresponds to the decay rate of our entropy according to Theorem \ref{thm:perturbation}, our choice of $\alpha^\ast\pa{\smin,\smax}$ was motivated by maximising the decay rate that is possible with our methodology. 
\end{remark}
We now posses all the tools which are required to prove part \eqref{item:convergence_sigma_not_constant} of Theorem \ref{thm:main}.
\begin{proof}[Proof of part \eqref{item:convergence_sigma_not_constant} of Theorem \ref{thm:main}]
The convergence estimation for $E_{\theta^\ast}\pa{u(t)-u_{\mathrm{avg}},v(t)}$ follows immediately from Theorem \ref{thm:perturbation} and Lemma \ref{lem:value_of_alpha}. To obtain \eqref{eq:main_for_f_pm_non_constant} we use Lemma \ref{lem:entropy_equivalence} in a similar fashion to the way we proved part \eqref{item:convergence_sigma_constant}.
\end{proof}

\section{Convergence to equilibrium in a $3-$velocity Goldstein-Taylor model}\label{sec:three-velocities}
The Goldstein-Taylor model can be thought of as a simplification of the BGK equation \cite{BGK, AAC}
$$\partial_{t}f(x,v,t)+ v\cdot \nabla _x f(x,v,t)- \nabla_x V(x)\cdot \nabla_{v}f(x,v,t)=M(v)\int f(x,v,t)dv - f(x,v,t),$$
where the variable $v$ is now in the discrete velocity space $\br{v_1,\dots,v_n}$, the variable $x$ is in the torus $\T$, {and} the potential $V(x)$ is zero. {The r.h.s.\ of the above BGK equation corresponds to a projection onto the Maxwellian $M(v)$; in the discrete velocity case this Maxwellian is replaced by a constant matrix that determines the large time behaviour of the new model.} Under the natural physical assumption of symmetry in the velocities (i.e. $\sum_{i=1}^n v_i=0$) and the expectation that the solutions will converge towards a state that is \textit{equally distributed in $v$} and constant in $x$ \footnote{
{If one wants to approximate the BGK equation with a Maxwellian relaxation function, then the column vector $(\frac1n,\dots,\frac1n)^T$ inside the relaxation matrix would have to be replaced by a {\it discrete Maxwellian}, as was done \cite[\S4.2]{Achleitner2018}}.
}, we find one potential multi-velocity extension of the Goldstein-Taylor model on $\T\times (0,\infty)$:
\begin{equation}\label{eq:gt_multi}
\partial_{t}\begin{pmatrix} f_1(x,t) \\ \vdots \\ f_n(x,t) \end{pmatrix}
+\mathcal V \begin{pmatrix} f_1(x,t) \\ \vdots \\ f_n(x,t) \end{pmatrix} = \sigma(x)\pa{\begin{pmatrix}  \frac{1}{n} \\ \vdots \\ \frac{1}{n} \end{pmatrix}\otimes \pa{1,\dots,1}-\II}\begin{pmatrix} f_1(x,t) \\ \vdots \\ f_n(x,t) \end{pmatrix},
\end{equation}
with the  the diagonal matrix $\mathcal V:=\mbox{diag}[v_1,\dots,v_n]$, and the discrete velocities
$$\br{v_1,\dots,v_n}=\begin{cases} \br{-k+\frac{1}{2},\dots,-\frac{1}{2},\frac{1}{2},\dots, k-\frac{1}{2}}, & n=2k \\ \br{-k,\dots, -1,0,1,\dots, k },& n=2k-1\end{cases},\quad n\in\N,\; n\geq 2.$$
The matrix on the r.h.s.\ of \eqref{eq:gt_multi} takes the form
$$ 
\bm{Q}=\frac{1}{n}\begin{pmatrix} 1-n & 1 & \dots & 1 \\ 1 & 1-n & \dots & 1 \\ \vdots & \vdots & \vdots & \vdots \\ 1 & 1 & \dots & 1-n
\end{pmatrix}$$
which has  
$\pa{1,1,\dots, 1}^T$ in its kernel, and $\mathcal{A}=\br{\pa{\xi_1,\dots, \xi_n}^T\in\R^n\;|\; \sum_{i=1}^n \xi_i=0}$ as its $n-1$ dimensional eigenspace corresponding to the eigenvalue $\lambda=-1$. This corresponds to the conservation of total mass, and the fact that differences such as $\br{f_i-f_j}_{i,j=1,\dots, n}$ converge to zero. For more information we refer the interested reader to \cite{AAC}. \\ 
In this section we will consider a simple $3-$velocity Goldstein-Taylor model, which is governed by the following system of equations on $\T\times (0,\infty)$
\begin{equation}\label{eq:gt_three}
\begin{gathered}
\partial_t f_1(x,t) + \partial_x f_1(x,t)=\frac{\sigma(x)}{3}\pa{f_2(x,t)+f_3(x,t)-2f_1(x,t)},\\
\partial_t f_2(x,t) =\frac{\sigma(x)}{3}\pa{f_1(x,t)+f_3(x,t)-2f_2(x,t)},\\
\partial_t f_3(x,t) - \partial_x f_3(x,t)=\frac{\sigma(x)}{3}\pa{f_1(x,t)+f_2(x,t)-2f_3(x,t)}.
\end{gathered}
\end{equation}
Much like our Goldstein-Taylor equation, \eqref{eq:gt}, we can recast the above with the variables
\begin{equation}\label{eq:f-u-transfo}
\begin{pmatrix} u_1 \\ u_2 \\ u_3 \end{pmatrix}=\begin{pmatrix}
\frac{1}{\sqrt{3}} & \frac{1}{\sqrt{3}} & \frac{1}{\sqrt{3}} \\
\frac{1}{\sqrt{2}} & 0 & -\frac{1}{\sqrt{2}} \\
\frac{1}{\sqrt{6}} & -\frac{2}{\sqrt{6}} & \frac{1}{\sqrt{6}}
\end{pmatrix} \begin{pmatrix}
f_1 \\ f_2 \\f_3
\end{pmatrix}\ ,
\end{equation}
which yields the following set of equations:
\begin{equation}\label{eq:gt_three_recast}
\begin{gathered}
\partial_t u_1(x,t) + \sqrt{\frac{2}{3}}\partial_x u_2(x,t)=0,\\
\partial_t u_2(x,t) +\sqrt{\frac{2}{3}} \partial_x u_1(x,t)+\frac{1}{\sqrt{3}}\partial_x u_3(x,t)=-\sigma(x) u_2(x,t),\\
\partial_t u_3(x,t) + \frac{1}{\sqrt{3}}\partial_x u_2(x,t)= - \sigma(x) u_3(x).
\end{gathered}
\end{equation}
{The} orthogonal transformation \eqref{eq:f-u-transfo} has a strong geometrical reasoning behind it, as it {diagonalises the appropriate ``interaction matrix'', $\bm{Q}$. It is also worth to mention that much like \eqref{eq:gt_recast}, this transformations brings us to the macroscopic variables. Indeed, up to some scaling $u_1$ is the mass, $u_2$ is the flux, and $u_3$ is a linear combination of the kinetic energy and the mass.

Following our intuition we expect that by denoting
$$u_\infty := \frac{1}{2\sqrt{3}\pi}\int_{\T}\pa{f_{1,0}(x)+f_{2,0}(x)+f_{3,0}(x)}dx,$$
we will find that 
$$u_1(t,x)\stackrel{t\to\infty}{\longrightarrow} u_\infty,\quad u_2(t,x)\stackrel{t\to\infty}{\longrightarrow} 0,\quad u_3(t,x)\stackrel{t\to\infty}{\longrightarrow} 0.$$
{To prove this result we shall introduce an appropriate Lyapunov functional. To find this functional, we have two options, even for the simple case of constant $\sigma$ (which is our base case): Proceeding as in \S \ref{sec4.2}, we could use a modal decomposition of \eqref{eq:gt_three_recast} and the (optimal) positive definite matrices $\PP_k$ to construct an entropy functional with sharp decay, and then rewrite it in physical space, using pseudo-differential operators. This construction, which is analogous to the construction of $E_\theta(f,g)$ from \eqref{eq:def_entropy}, can become extremely cumbersome in dimension 3 and higher. \\
As a simpler alternative we shall hence rather follow the strategy from \cite[\S4.3]{AAC} and \cite[\S 2.3]{Achleitner2018}: In Fourier space, the system matrix of \eqref{eq:gt_three_recast} reads as
$$
  \C_k=\begin{pmatrix}
0 & \sqrt{\frac{2}{3}}ik & 0 \\
\sqrt{\frac{2}{3}}ik & \sigma & {\frac{1}{\sqrt 3}}ik \\
0 & {\frac{1}{\sqrt 3}}ik & 0
\end{pmatrix}\ .
$$
We note that, for $k\ne0$, the \textit{hypocoercivity index}\footnote{This index  characterises the degree of degeneracy of ODE or PDE-evolution equations, cf.\ \cite{Achleitner2018}.} of $\C_k$, as well as of \eqref{eq:gt_three_recast} is one, since this index is always bounded from above by the kernel dimension of the symmetric part of the generator, cf.\ \cite{Achleitner2018}. For such index-1 problems, Theorem 2.6 from \cite{Achleitner2018} shows that the choice
$$
  \PP_k=\begin{pmatrix}
1 & \frac{\lambda}{ik} & 0 \\
-\frac{\lambda}{ik} & 1 & 0 \\
0 & 0 & 1
\end{pmatrix}\ \quad k\ne0,
$$
with an appropriate $\lambda\in\R$, always yields a (simple) Lyapunov functional for \eqref{eq:gt_three_recast}, typically with a sub-optimal decay rate. Much like in \S \ref{sec:sigma_general}, this guides us to the definition of our functional, expressed in the following theorem: 
} 

\begin{theorem}\label{thm:three-velocities}
Let $u_1,u_2,u_3\in C([0,\infty);L^2\pa{\T})$ be mild real valued solutions to  \eqref{eq:gt_three_recast} with initial datum $u_{1,0},u_{2,0},u_{3,0}\in L^2\pa{\T}$. Denoting by
$$\mathfrak{E}_{\theta}\pa{f,g,h}:=\norm{f}^2+\norm{g}^2+\norm{h}^2 - \frac{\theta}{2\pi}\int_{0}^{2\pi}\pa{\ad f(x) \,{g(x)}}dx,$$
we have that
\begin{equation}\label{eq:three_velocities}
\mathfrak{E}_{\theta}\pa{u_1(t)-u_{\infty},u_2(t),u_3(t)}\leq \mathfrak{E}_{\theta}\pa{u_{1,0}-u_{\infty},u_{2,0},u_{3,0}} e^{-\alpha t}\ ,\quad t\ge0,
\end{equation}
for any $\theta>0$ and $\alpha>0$ such that
\begin{equation}\label{eq:three_velocities_conditions_I}
\sqrt{\frac{2}{3}}\theta +\alpha < 2\smin, \quad \alpha \leq \sqrt{\frac{2}{3}}\theta,
\end{equation}
and
\begin{equation}\label{eq:three_velocities_conditions_II}
\pa{\sup_{x\in\T}\frac{\theta^2\pa{\sigma(x)-\alpha}^2}{8\sigma(x)-4\sqrt{\frac{2}{3}}\theta-4\alpha}}+\pa{\sup_{x\in\T}\frac{\theta^2}{12\pa{2\sigma(x)-\alpha}}} \leq \sqrt{\frac{2}{3}}\theta-\alpha.
\end{equation}
\end{theorem}
\begin{remark}\label{rem:weighted_norms}
{For $0<\theta<2$, $\mathfrak{E}_{\theta}\pa{f,g,h}$ is equivalent to $\norm{f}^2+\norm{g}^2+\norm{h}^2$. Indeed, following Lemma \ref{lem:entropy_equivalence} we see that }
\begin{equation}\nonumber
\begin{split}
\pa{1-\frac{\abs{\theta}}{2}}\pa{\norm{f}^2+\norm{g}^2}+\norm{h}^2\leq \mathfrak{E}_{\theta}\pa{f,g,h} \leq \pa{1+\frac{\abs{\theta}}{2}}\pa{\norm{f}^2+\norm{g}^2}+\norm{h}^2.
\end{split}
\end{equation}
\end{remark}
\begin{proof}[Proof of Theorem \ref{thm:three-velocities}]
We start by noticing that the transformation
$$u_1\to u_1-u_\infty,\quad u_2\to u_2,\quad u_3\to u_3$$
keeps \eqref{eq:gt_three_recast} invariant, so we may assume, without loss of generality, that $u_\infty=0$. This, together with the equation for $u_1(x,t)$ implies that
$$\pa{u_1(t)}_{\mathrm{avg}}=\pa{u_{1,0}}_{\mathrm{avg}}=u_{\infty}=0.$$
Next, we compute the time derivatives of the $L^2$ norms and obtain:
\begin{equation}\label{eq:three_velocities_proof_I}
\begin{gathered}
\frac{d}{dt}\pa{\norm{u_1(t)}^2 + \norm{u_2(t)}^2+\norm{u_3(t)^2}}=-\frac{1}{\pi}\int_0^{2\pi}\sigma(x)u_2(x,t)^2dx\\
-\frac{1}{\pi}\int_0^{2\pi}\sigma(x)u_3(x,t)^2dx.
\end{gathered}
\end{equation}
Continuing, we see that 
\begin{equation}\nonumber
\begin{gathered}
\frac{d}{dt}\int_{0}^{2\pi} \ad u_1(x,t) u_2(x,t)dx
=2\pi \sqrt{\frac{2}{3}}\pa{\pa{u_2(t)_{\mathrm{avg}}}^2-\norm{u_2(t)}^2} + \frac{2\sqrt{2}\pi}{\sqrt{3}}\norm{u_1(t)}^2
\end{gathered}
\end{equation}
\begin{equation}\nonumber
\begin{gathered}
+\frac{1}{\sqrt{3}}\int_{0}^{2\pi}  u_1(x,t)u_3(x,t)dx-\int_{0}^{2\pi} \sigma(x)\ad u_1(x,t) u_2(x,t)dx,
\end{gathered}
\end{equation}
where we used Lemma \ref{lem:properties_of_antiderivative}.
As such,
together with \eqref{eq:three_velocities_proof_I}, we conclude that 
\begin{equation}\label{eq:three_velocities_proof_II}
\begin{gathered}
\frac{d}{dt}\mathfrak{E}_{\theta}\pa{u_1(t),u_2(t),u_3(t)}=-\frac{1}{2\pi}\int_0^{2\pi}\pa{2\sigma(x)-\sqrt{\frac{2}{3}}\theta}u_2(x,t)^2dx\\
-\frac{1}{\pi}\int_0^{2\pi}\sigma(x)u_3(x,t)^2dx-\sqrt{\frac{2}{3}}\theta\norm{u_1(t)}^2-\sqrt{\frac{2}{3}}\theta\pa{u_2(t)_{\mathrm{avg}}}^2\\
-\frac{\theta}{2\sqrt{3}\pi}\int_{0}^{2\pi}  u_1(x,t)u_3(x,t)dx+\frac{\theta}{2\pi}\int_{0}^{2\pi} \sigma(x)\ad u_1(x,t) u_2(x,t)dx.
\end{gathered}
\end{equation}
Thus
\begin{equation}\nonumber
\frac{d}{dt}\mathfrak{E}_{\theta}\pa{u_1(t),u_2(t),u_3(t)}=-\alpha \mathfrak{E}_{\theta}\pa{u_1(t),u_2(t),u_3(t)}+R_{\theta,\alpha,\sigma}(t)
\end{equation}
with
\begin{equation}\label{eq:def_of_R}
\begin{gathered}
R_{\theta,\alpha,\sigma}(t):=-\frac{1}{2\pi}\int_0^{2\pi}\pa{2\sigma(x)-\sqrt{\frac{2}{3}}\theta-\alpha}u_2(x,t)^2dx\\
-\frac{1}{2\pi}\int_0^{2\pi}\pa{2\sigma(x)-\alpha}u_3(x,t)^2dx-\pa{\sqrt{\frac{2}{3}}\theta-\alpha}\norm{u_1(t)}^2-\sqrt{\frac{2}{3}}\theta\pa{u_2(t)_{\mathrm{avg}}}^2\\
-\frac{\theta}{2\sqrt{3}\pi}\int_{0}^{2\pi}  u_1(x,t)u_3(x,t)dx+\frac{\theta}{2\pi}\int_{0}^{2\pi} \pa{\sigma(x)-\alpha}\ad u_1(x,t) u_2(x,t)dx.
\end{gathered}
\end{equation}
To conclude the proof it is enough to show that under conditions \eqref{eq:three_velocities_conditions_I} and \eqref{eq:three_velocities_conditions_II} we have that $R_{\theta,\alpha,\sigma}(t) \leq 0$. We will, in fact, show the stronger statement:
\begin{equation}\label{eq:three_velocities_proof_III}
\begin{gathered}
\abs{-\frac{\theta}{2\sqrt{3}\pi}\int_{0}^{2\pi}  u_1(x,t)u_3(x,t)dx+\frac{\theta}{2\pi}\int_{0}^{2\pi} \pa{\sigma(x)-\alpha}\ad u_1(x,t) u_2(x,t)dx}\\
\leq \frac{1}{2\pi}\int_0^{2\pi}\pa{2\sigma(x)-\sqrt{\frac{2}{3}}\theta-\alpha}u_2(x,t)^2dx\\
+\frac{1}{2\pi}\int_0^{2\pi}\pa{2\sigma(x)-\alpha}u_3(x,t)^2dx+\pa{\sqrt{\frac{2}{3}}\theta-\alpha}\norm{u_1(t)}^2.
\end{gathered}
\end{equation}
Similarly to the techniques we have used in the proof of part \eqref{item:convergence_sigma_not_constant} of Theorem \ref{thm:main}, and using the positivity of the coefficients in the last two terms (which follows from \eqref{eq:three_velocities_conditions_I}), we see that
$$\abs{\frac{\theta}{2\pi}\int_{0}^{2\pi} \pa{\sigma(x)-\alpha}\ad u_1(x,t) u_2(x,t)dx}$$
$$\leq \frac{\theta}{2\pi}\int_{0}^{2\pi}\frac{\abs{\sigma(x)-\alpha}}{\sqrt{2\sigma(x)-\sqrt{\frac{2}{3}}\theta-\alpha}}\abs{\ad u_1(x,t)} \cdot\sqrt{2\sigma(x)-\sqrt{\frac{2}{3}}\theta-\alpha}\abs{u_2(x,t)}dx$$
$$\leq \pa{\sup_{x\in\T}\frac{\theta^2\pa{\sigma(x)-\alpha}^2}{8\sigma(x)-4\sqrt{\frac{2}{3}}\theta-4\alpha}} \norm{u_1(t)}^2+ \frac{1}{2\pi}\int_0^{2\pi}\pa{2\sigma(x)-\sqrt{\frac{2}{3}}\theta-\alpha}u_2(x,t)^2dx,$$
and that
$$\abs{\frac{\theta}{2\sqrt{3}\pi}\int_{0}^{2\pi}  u_1(x,t)u_3(x,t)dx} \leq \frac{\theta}{2\pi}\int_{0}^{2\pi}  \frac{\abs{u_1(x,t)}}{\sqrt{6\sigma(x)-3\alpha}}\sqrt{2\sigma(x)-\alpha}\abs{u_3(x,t)}dx$$
$$\leq \pa{\sup_{x\in\T}\frac{\theta^2}{12\pa{2\sigma(x)-\alpha}}}\norm{u_1(t)}^2+\frac{1}{2\pi}\int_0^{2\pi}\pa{2\sigma(x)-\alpha}u_3(x,t)^2dx.$$
Thus, one sees that \eqref{eq:three_velocities_proof_III} holds when
$$\pa{\sup_{x\in\T}\frac{\theta^2\pa{\sigma(x)-\alpha}^2}{8\sigma(x)-4\sqrt{\frac{2}{3}}\theta-4\alpha}}+\pa{\sup_{x\in\T}\frac{\theta^2}{12\pa{2\sigma(x)-\alpha}}} \leq \sqrt{\frac{2}{3}}\theta-\alpha,$$
which is \eqref{eq:three_velocities_conditions_II}. The proof is complete.
\end{proof}
While we have elected not to optimise the choice of $\alpha$ (as in \S\ref{sec:sigma_general}), we can still infer the following, simpler yet far from optimal, corollary: 
\begin{corollary}\label{cor:three_velocities_estimation}
Let $\theta>0$ and $\alpha>0$ be such that
\begin{equation}\nonumber
\sqrt{\frac{2}{3}}\theta +\alpha < 2\smin, \quad \alpha \leq \sqrt{\frac{2}{3}}\theta
\end{equation}
and
\begin{equation}\label{eq:cond-simple}
\frac{\theta^2\smax^2}{8\smin-4\sqrt{\frac{2}{3}}\theta-4\alpha}+\frac{\theta^2}{12\pa{2\smin-\alpha}} \leq \sqrt{\frac{2}{3}}\theta-\alpha.
\end{equation}
then 
\begin{equation}\nonumber
\mathfrak{E}_{\theta}\pa{u_1(t)-u_{\infty},u_2(t),u_3(t)}\leq \mathfrak{E}_{\theta}\pa{u_{1,0}-u_{\infty},u_{2,0},u_{3,0}} e^{-\alpha t}.
\end{equation}
In particular, for  
$$\alpha := \min\pa{\frac{\smin}{2},\frac{3\smin}{9\smax^2+1}}$$
we have that {$\mathfrak{E}_{\sqrt{6}\alpha}$} decays exponentially to zero with rate $\alpha$. 
\end{corollary}
\begin{proof}
Since $\alpha<2\smin\leq \smax+\smin$ we see that 
$$\alpha - \smax < \smin \leq \sigma(x) <\smax + \alpha,$$
implying that $\pa{\sigma(x)-\alpha}^2 \leq \smax^2$ for any $x\in\T$. Using this with additional elementary estimation on the denominator of the expressions that appear in \eqref{eq:three_velocities_conditions_II}, we see that \eqref{eq:three_velocities_conditions_I} and \eqref{eq:three_velocities_conditions_II} are valid. As such the first statement of the corollary follows from Theorem \ref{thm:three-velocities}. \\
To show the second part of the corollary we notice that with the choice $\theta_{\alpha}:=\sqrt{6}\alpha$ and $\alpha\leq \frac{\smin}{2}$
$$\sqrt{\frac{2}{3}}\theta_\alpha +\alpha=3\alpha< 2\smin,\quad\alpha \leq 2\alpha=\sqrt{\frac{2}{3}}\theta_{\alpha},$$
giving us \eqref{eq:three_velocities_conditions_I}. 
Using the inequalities
$$8\smin-4\sqrt{\frac{2}{3}}\theta_\alpha-4\alpha\geq 2\smin,\quad\text{and}\quad 2\smin-\alpha \geq  \frac{3}{2}\smin$$
for the l.h.s.\ of \eqref{eq:cond-simple}, 
we see that
$$\frac{\theta_\alpha^2\smax^2}{8\smin-4\sqrt{\frac{2}{3}}\theta_\alpha-4\alpha}+\frac{\theta_\alpha^2}{12\pa{2\smin-\alpha}} \leq \pa{9\smax^2+1}\frac{\alpha^2}{3\smin}.
$$
Thus, since $\sqrt{\frac{2}{3}}\theta_\alpha-\alpha=\alpha$, the desired condition \eqref{eq:cond-simple} is valid when 
$$\alpha \leq  \frac{3\smin}{9\smax^2+1},$$
which concludes the proof.
\end{proof}


\appendix
\section{Lack of optimality}\label{appsec:lack}
In this appendix we will briefly discuss the lack of optimality of our decay rate for non-homogeneous $\sigma(x)$ in comparison to that given in \cite{BS}. We will even go one step further and show how one can improve our general methodology in simple cases, though even this improvement will fall short of the optimal convergence rate.

As one simple example we will explore the following relaxation function:
\begin{equation}\label{appeq:choice_of_sigma}
\sigma(x):=\begin{cases} 1, & 0 < x \leq \pi \\ 4, & \pi<x\leq 2\pi \end{cases},
\end{equation}
which is motivated by the fact that for this function $\smin=\frac{\smax}{4}$, and so the choice of $\theta^\ast=1$ in our main Theorem \ref{thm:main} \eqref{item:convergence_sigma_not_constant} comes ``from both directions''.\\
Before we start with a more structured discussion, we would like to explain how one can improve the technique we developed in \S\ref{sec:sigma_general}. 
A crucial point in the investigation of the behaviour of $E_{\theta^\ast}$ was to find, and close, a linear differential inequality for this entropy, as can be seen in the proof of Theorem \ref{thm:perturbation}.
 One of the final steps in this proof, appearing in \eqref{eq:perturbation_for_app}, was to combine Poincar\'e inequality with an 
$L^\infty$ estimation on the mixed term of $\ad \pa{u-u_{\mathrm{avg}}}$ and $v$, to show the non-positivity of an appropriate ``remainder''. 
The use of these two inequalities is somewhat crude (yet due to that, quite general), and one can imagine that replacing these two estimations with an inequality that is more $L^2$ based  would improve the range of validity of the theorem. One idea that comes to mind is a \textit{weighted Poincar\'e inequality}, i.e. an inequality of the form 
\begin{equation}\label{appeq:weighted_poincare}
\int_{0}^{2\pi} \pa{f(x)-f_{\mathrm{avg}}}^2 \omega(x)dx \leq C^2\int_{0}^{2\pi} \pa{f^\prime(x)}^2dx.
\end{equation}
for a given weight $\omega(x)\ge0$ and constant $C$. Denoting by\footnote{Note that by definition, and by Lemma \ref{thm:poincare}, $C_\omega \leq \sqrt{\norm{\omega}_{\infty}}$, and as such we will automatically get an improvement to condition \eqref{eq:perturbative_condition_II}. }
$$C_{\omega}:=\inf \br{C>0\;\Big|\;\int_{0}^{2\pi} \pa{f(x)-f_{\mathrm{avg}}}^2 \omega(x)dx \leq C^2\int_{0}^{2\pi} \pa{f^\prime(x)}^2dx,\quad\forall f\in H^1(\T)},$$ 
we can replace condition \eqref{eq:perturbative_condition_II} of Theorem \ref{thm:perturbation} with the improved condition
\begin{equation}\label{appeq:better_condition}
\frac{\theta^2}{4}C^2_{\omega} \leq \theta-\alpha,\quad 
\text{where}\quad \omega(x)=\frac{\pa{\sigma(x)-\alpha}^2}{2\sigma(x)-\theta-\alpha}.
\end{equation}
\eqref{appeq:better_condition} will be explicitly derived in \S\ref{appsubsec:improved}.

{}From this point onwards the appendix will proceed as follows: First we will show how one can find the optimal weighted Poincar\'e constant, and compute it in some simple cases, which we will then use in the case were $\sigma(x)$ is given by \eqref{appeq:choice_of_sigma} to obtain an improvement of our current rate of convergence to equilibrium. Next we will compute the optimal rate given by \cite{BS}, and conclude a lack of optimality by comparing the rate we achieved in our main theorem, the improved rate we have found, and the optimal rate of \cite{BS}.

\subsection{Weighted Poincar\'e inequality}\label{appsubsec:weighted_poincare} The problem of finding a weighted Poincar\'e inequality and its associated sharp constant can be recast as a constrained variational problem. We define the functional
$$\F:\;\;\;\;\D:=H^1\pa{\T}\to \R,$$
where $H^1\pa{\T}$ is the Sobolev space of real valued periodic functions, by
$$\F(u):=\int_0^{2\pi}\pa{u^\prime(x)}^2dx,$$
and denote by
\begin{equation}\label{appeq:minimisation}
c_{\mathrm{min}}:=\inf\br{\F(u)\;\Big|\;u\in\D,\;\int_{0}^{2\pi}u(x)^2\omega(x)dx=1,\;\int_0^{2\pi}u(x)dx=0}.
\end{equation}
Even though the minimization set is not convex, standard techniques from Calculus of Variation (see for instance \cite[\S 8]{Ev}) show that if $\omega$ is bounded then the infimum is attained 
(the conditions on $\omega$ can be weakened).\\
One can easily check that in that case
$$C^2_{\omega}=\frac{1}{c_{\mathrm{min}}}.$$
Finding a minimiser to the problem \eqref{appeq:minimisation} amounts to solving the following constrained Euler-Lagrange equation on $\T$
\begin{equation}\label{appeq:ODE_for_minimiser}
u^{\prime\prime}(x)+\lambda u(x)\omega(x)-\tau=0,
\end{equation}
considered in weak form,  with two Lagrange multipliers $\lambda>0$ and $\tau\in\R$. Integrating \eqref{appeq:ODE_for_minimiser} against $u$ shows that
\begin{equation}\label{appeq:lambda_min}
\lambda=\F(u),
\end{equation}
which we will use shortly.

Since $\omega\in L^\infty(\T)$, we find that $u\in H^2(\T)\hookrightarrow C^1(\T)$. When $\omega$ is piecewise constant, the ODE \eqref{appeq:ODE_for_minimiser}, now in strong form, can be solved explicitly. This shows that in these cases the minimiser of $\F$ is actually unique.

As we shall see in \S\ref{appsubsec:improved} below, the relevant weight functions we require for our improved study are closely related to $\sigma(x)$. With \eqref{appeq:choice_of_sigma} in mind, we shall consider weights of the form:
$$\omega(x):=\begin{cases} \omega_1, & 0< x\leq \pi \\ \omega_2, & \pi<x\leq 2\pi \end{cases}.
$$
Hence, the solution to the Euler-Lagrange equation \eqref{appeq:ODE_for_minimiser} is given by
\begin{equation}\label{eqapp:solution_euler_lagrange}
\begin{gathered}
u(x)=\begin{cases} c_1\sin\pa{\sqrt{\lambda \omega_1}x}+c_2\cos\pa{\sqrt{\lambda \omega_1}x} + \frac{\tau}{\lambda \omega_1}, & 0<x<\pi \\ 
c_3\sin\pa{\sqrt{\lambda \omega_2}x}+c_4\cos\pa{\sqrt{\lambda \omega_2}x} + \frac{\tau}{\lambda \omega_2},& \pi<x<2\pi\end{cases}\\
=:\begin{cases} u_1(x),  & 0<x<\pi \\ u_2(x), & \pi<x<2\pi\end{cases},
\end{gathered}
\end{equation}
and it satisfies the following $C^1$-matching conditions and constraints:
\begin{equation}\label{eqapp:euler_lagrange_BD}
\begin{gathered}
u_1(0)=u_2\pa{2\pi},\\
u_1(\pi)=u_2(\pi),\\
u_1^\prime(0)=u_2^\prime(2\pi),\\
u_1^\prime(\pi)=u_2^\prime(\pi),\\
\int_{0}^{\pi} u_1(x)dx+\int_{\pi}^{2\pi} u_2(x)dx=0,\\
\int_{0}^{\pi} \omega_1u_1(x)^2dx+\int_{\pi}^{2\pi} \omega_2 u_2^2(x)dx=1.
\end{gathered}
\end{equation}
The first five equations correspond to the linear set of equations:
$$\bm{M}(\lambda)\begin{pmatrix}c_1 \\ c_2 \\ c_3 \\ c_4 \\ \tau \end{pmatrix}=\begin{pmatrix}0\\0\\0\\0 \\ 0\end{pmatrix}\ ,$$
where the matrix $\bm{M}\pa{\lambda}$ is
$$\tiny \begin{pmatrix}
0 & 1 & -\sin\pa{2\pi\sqrt{\lambda \omega_2}} & -\cos\pa{2\pi\sqrt{\lambda \omega_2}} & \frac{\omega_2-\omega_1}{\lambda\omega_1\omega_2} \\
\sin\pa{\pi\sqrt{\lambda \omega_1}} & \cos\pa{\pi\sqrt{\lambda \omega_1}} & -\sin\pa{\pi\sqrt{\lambda \omega_2}} & -\cos\pa{\pi\sqrt{\lambda \omega_2}}  & \frac{\omega_2-\omega_1}{\lambda\omega_1\omega_2} \\
\sqrt{\omega_1} & 0 & -\sqrt{\omega_2}\cos\pa{2\pi\sqrt{\lambda \omega_2}} & \sqrt{\omega_2}\sin\pa{2\pi\sqrt{\lambda \omega_2}} & 0\\
\sqrt{\omega_1}\cos\pa{\pi\sqrt{\lambda \omega_1}} & -\sqrt{\omega_1}\sin\pa{\pi\sqrt{\lambda \omega_1}} & -\sqrt{\omega_2}\cos\pa{\pi\sqrt{\lambda \omega_2}} & \sqrt{\omega_2}\sin\pa{\pi\sqrt{\lambda \omega_2}} & 0\\
\frac{1-\cos\pa{\pi\sqrt{\lambda \omega_1}}}{\sqrt{\omega_1}} & \frac{\sin\pa{\pi\sqrt{\lambda\omega_1}}}{\sqrt{\omega_1}} & \frac{\cos\pa{\pi\sqrt{\lambda \omega_2}}-\cos\pa{2\pi\sqrt{\lambda \omega_2}}}{\sqrt{\omega_2}} & \frac{\sin\pa{2\pi\sqrt{\lambda\omega_2}}-\sin\pa{\pi\sqrt{\lambda\omega_2}}}{\sqrt{\omega_2}} & \frac{\pi\pa{\omega_2+\omega_1}}{\sqrt{\lambda}\omega_1\omega_2}
\end{pmatrix}.$$
As we are looking for a non-zero solution to the above equation, we must have that $\text{det}\pa{\bm{M}(\lambda)}=0$. In \eqref{eqapp:euler_lagrange_BD}, the last condition on $u_1$ and $u_2$ merely acts as normalisation, and doesn't help in finding $\lambda$. Hence, due to \eqref{appeq:lambda_min}, we find that
$$c_{\mathrm{min}}\pa{\omega_1,\omega_2}=\min\br{\lambda>0\;|\;\text{det}\pa{\bm{M}(\lambda)}=0}.$$
This is how we can find $c_{\mathrm{min}}$, and consequently $C_{\omega}$, explicitly (numerically in many cases).

\subsection{Improved methodology}\label{appsubsec:improved} We return now to the proof of the differential inequality \eqref{E-diff-inequ} that governs the evolution of $E_{\theta}$, which is essentially based on the estimate \eqref{eq:perturbation_for_app}. Choosing $\theta^\ast=1$ and $\sigma(x)$ as in \eqref{appeq:choice_of_sigma} and using the weight
\begin{equation}\nonumber
\omega_{\alpha}(x):=\frac{\pa{\sigma(x)-\alpha}^2}{2\sigma(x)-1-\alpha}=\begin{cases} 1-\alpha & 0< x\leq \pi \\ \frac{\pa{4-\alpha}^2}{7-\alpha} & \pi<x\leq 2\pi \end{cases},
\end{equation}
which appears in the penultimate line of \eqref{eq:perturbation_for_app}, we see that by using the previously discussed weighted Poincar\'e inequality instead of the last step of \eqref{eq:perturbation_for_app}, we obtain from \eqref{E-diff-inequ} and \eqref{eq:perturbation_for_app}:
\begin{equation}\label{E-diff-inequ-new}
\begin{gathered}
\frac{d}{dt}E_1\pa{u(t)-u_{\mathrm{avg}},v(t)} \leq  - \alpha E_1\pa{u(t)-u_{\mathrm{avg}},v(t)} \\
-\pa{1-\alpha-\frac{C^2_{\omega_{\alpha}}}{4}}\norm{u(t)-u_{\mathrm{avg}}}^2.
\end{gathered}
\end{equation}
We will maximise the decay rate $\alpha$, satisfying
\begin{equation}\label{eqapp:improved_main_condition}
0<\alpha \leq 1-\frac{C^2_{\omega_\alpha}}{4}<1
\end{equation}
(so that the second term in \eqref{E-diff-inequ-new} is non-positive) by a processes of iteration:
Guessing the starting value $\alpha_0:=\alpha^{\ast}\pa{1,4}=2\pa{2-\sqrt{3}}$ (the rate one obtains from our main Theorem \ref{thm:main}, cf.\ \eqref{eq:value_of_alpha}) we follow the process described in the previous subsection and find the weighted Poincar\'e constant $C^2_{\omega_{\alpha_0}}=1.12013...$, which indeed satisfies \eqref{eqapp:improved_main_condition}.\\
We proceed and create a sequence $\br{\alpha_n}_{n\in\N_0}$, defined recursively, so that each $\alpha_n$ improves upon the previous step by taking its ``optimal'' value, i.e.
$$\alpha_{n}:=1-\frac{C^2_{\omega_{\alpha_{n-1}}}}{4},\quad\quad n\in\N,$$
as long as \eqref{eqapp:improved_main_condition} is still satisfied {for this choice}. A change of $\alpha$ implies a change of our weight function $\omega_\alpha(x)$, yet these new weights are still of the form given in our previous subsection. As such we are able to compute the appropriate $C_{\omega_{\alpha_n}}\;'$s, and to show that this sequence converges to the improved decay rate\footnote{This process was dealt with numerically.} 
$$\alpha_{\mathrm{max}}\approx 0.7234.$$ 


\subsection{Comparison of convergence rates}\label{appsubsec:optiimal_rate} 
The optimal rate\footnote{at least for $H^1$-initial data} of exponential convergence to the Goldstein-Taylor equation, \eqref{eq:gt}, was found by Bernard and Salvarani in \cite{BS}. Taking into account the different scaling of the torus $\T$ in our paper, this convergence rate is given by
$$\alpha_{\mathrm{BS}}:=\frac{1}{\pi}\min\pa{\norm{\widetilde{\sigma}}_{L^1\pa{\frac{\T}{2\pi}}}, \tilde{D}(0)}$$
where
$$\widetilde{\sigma}(\xi):=\pi \sigma\pa{2\pi \xi},\quad\quad \xi\in \frac{\T}{2\pi},$$
and $\tilde{D}(0)$ is the spectral gap of the telegrapher's equation, see \cite[Proposition 3.5]{BS}, \cite[Theorem 2]{L}. More precisely, 
$$
\tilde{D}(0):=\inf\br{\Re \lambda_j\,\big|\,\lambda_j\in \Big(\mbox{ spectrum of }\bm{A}_{\tilde{\sigma}}=\begin{pmatrix}
0 & -1\\
-\partial_{xx} & 2\tilde{\sigma}
\end{pmatrix} \mbox{ in } H^2\oplus H^1\Big)
\setminus\{0\} }.
$$ 
We note that, for $\tilde{\sigma}$ constant and in Fourier space, the matrix $\begin{pmatrix}
0 & -1\\
k^2 & 2\tilde{\sigma}
\end{pmatrix}$ is related to $\C_k$ from \eqref{ODEsystem} by a simple similarity transformation.

Following on our choice for $\sigma(x)$ from \eqref{appeq:choice_of_sigma}, we see that
\begin{equation}\label{appeq:specific_sigma_tilde}
\widetilde{\sigma}(\xi)=\begin{cases} \sigma_1:=\pi, & 0 < \xi \leq \frac{1}{2} \\  \sigma_2:=4\pi, & \frac{1}{2}<\xi\leq 1 \end{cases},
\end{equation}
and as such $\norm{\widetilde{\sigma}}_{L^1\pa{\frac{\T}{2\pi}}}=\frac{5\pi}{2}$.\\
The calculation of $\tilde{D}(0)$ is more involved. 
According to \cite{L}, the spectrum of $\bm{A}_{\widetilde{\sigma}}$, besides potentially $\br{0}$, is discrete and the real part of its eigenvalues must lie in {$(0,2\norm{\widetilde{\sigma}}_{\infty}]$.} {A more detailed investigation of the spectrum can be found in \cite{CZ}.}\\
The eigenvalue problem
$$\bm{A}_{\widetilde{\sigma}}\begin{pmatrix}u \\v\end{pmatrix}= \gamma \begin{pmatrix}u \\v\end{pmatrix},$$
with $\gamma\not=0$, is equivalent to the set of equations
\begin{equation}\nonumber
v^{\prime\prime}(\xi)=\gamma\pa{\gamma - 2\widetilde{\sigma}(\xi)}v(\xi), \quad v(\xi)=-\gamma u(\xi).
\end{equation}
{To find $\tilde D(0)$ it is sufficient to consider only eigenvalues with $\Re \gamma \in (0,2\sigma_1)=(0,2\pi)$, since this complex strip already includes one (real) eigenvalue as we shall see below.}
With the notation $\tau_{1,2}(\gamma):=\sqrt{\gamma(2\sigma_{1,2}-\gamma)}$, which may have to be considered as a complex root, the solution of the ODE is of the form
$$
  v(\xi)=\begin{cases} A_1 \cos(\tau_1(\gamma)\,\xi) +B_1\sin(\tau_1(\gamma)\,\xi), & 0 < \xi \leq \frac{1}{2} \\  
A_2 \cos(\tau_2(\gamma)\,\xi) +B_2\sin(\tau_2(\gamma)\,\xi), & \frac{1}{2}<\xi\leq 1 \end{cases}.
$$
With $C^1$--matching conditions at $\xi=0$ and $\xi=\frac12$, the coefficients satisfy the following system of linear equations:
$$\bm{M}(\gamma)\begin{pmatrix}A_1 \\ B_1 \\ A_2 \\ B_2\end{pmatrix}=\begin{pmatrix}0\\0\\0\\0\end{pmatrix}$$
where the matrix $\bm{M}(\gamma)$ is given by
\begin{equation}\nonumber
 \left( \begin{tabular}{cccc}
$1$ & $0$ & $-\cos\pa{\tau_2(\gamma)}$ & $-\sin\pa{\tau_2(\gamma)}$\\
$0$ & 1 &  $\frac{\tau_2(\gamma)}{\tau_1(\gamma)}\sin\pa{\tau_2(\gamma)}$ & $-\frac{\tau_2(\gamma)}{\tau_1(\gamma)}\cos\pa{\tau_2(\gamma)}$\\ 
$\cos\pa{\frac{\tau_1(\gamma)}{2}}$ & $\sin\pa{\frac{\tau_1(\gamma)}{2}}$ & $-\cos\pa{\frac{\tau_2(\gamma)}{2}}$ & $-\sin\pa{\frac{\tau_2(\gamma)}{2}}$\\[1mm]
$\sin\pa{\frac{\tau_1(\gamma)}{2}}$ & $-\cos\pa{\frac{\tau_1(\gamma)}{2}}$ &  $-\frac{\tau_2(\gamma)}{\tau_1(\gamma)}\sin\pa{\frac{\tau_2(\gamma)}{2}}$ & $\frac{\tau_2(\gamma)}{\tau_1(\gamma)}\cos\pa{\frac{\tau_2(\gamma)}{2}}$
\end{tabular}  \right) .
\end{equation}
The requirement that 
\begin{eqnarray}\label{detM}
  \text{det}\pa{\bm{M}(\gamma)}&=&-\sin\pa{\frac{\tau_1(\gamma)}{2}}\sin\pa{\frac{\tau_2(\gamma)}{2}}
  \pa{1+\pa{\frac{\tau_2(\gamma)}{\tau_1(\gamma)}}^2}\\
  &+& 2\frac{\tau_2(\gamma)}{\tau_1(\gamma)} \pa{\cos\pa{\frac{\tau_1(\gamma)}{2}}\cos\pa{\frac{\tau_2(\gamma)}{2}}-1}=0\nonumber
\end{eqnarray}
yields the wanted eigenvalues $\gamma\in\CC$ {with $\Re\gamma\in(0,2\pi)$}.\\ 
{In our case, i.e. when $\tilde\sigma(x)$ is given by \eqref{appeq:specific_sigma_tilde}, we find (numerically) that the minimal real part of the non-zero eigenvalues found from \eqref{detM} 
is approximately} $ 2.72831$, which implies that $\tilde{D}(0)\approx 2.72831$. Thus, the optimal decay rate given by \cite{BS} is
$$\alpha_{\mathrm{BS}}\approx \frac{1}{\pi}\min\pa{\frac{5\pi}{2},2.72831}\approx 0.86845 .$$
\medskip

Summarising, we now have three convergence rates for the case 
\begin{equation}\nonumber
\sigma(x)=\begin{cases} 1, & 0\leq x \leq \pi \\ 4, & \pi<x\leq 2\pi \end{cases}\quad :
\end{equation}
\begin{itemize}
\item Rate from our main Theorem \ref{thm:main}: $\alpha^{\ast}=4-\sqrt{12}\approx 0.5359$.
\item Rate from our improved technique in \S\ref{appsubsec:improved}: $\alpha_{\mathrm{max}}\approx 0.7234$.
\item Rate from the work of Bertrand and Salvarani: $\alpha_{\mathrm{BS}}\approx  0.86845$.
\end{itemize}
This shows, as expected, the lack of optimality in our technique. 


\section{Deferred proofs}\label{appsec:proofs}

\begin{proof}[Proof of Lemma \ref{thm:poincare}]
While the proof is standard, we show it here for completion and to fix the sharp constant. Denoting the $k-$th Fourier coefficient of $f$ by 
$$\widehat{f}(k):=\frac{1}{2\pi}\int_{0}^{2\pi}f(x)e^{-i kx}dx$$
we find that
\begin{equation}\label{eq:fourier_of_derivative}
\widehat{f^\prime}(k)=ik\widehat{f}(k)
\end{equation}
for all $k\in\Z$ (including $k=0$). The condition $f_{\mathrm{avg}}=0$ is equivalent to $\widehat{f}(0)=0$ and as such, using Plancherel's equality, we find that
$$\norm{f}^2=\sum_{k\in\Z}\abs{\widehat{f}(k)}^2=\sum_{k\in\Z\setminus\br{0}}\abs{\widehat{f}(k)}^2$$
$$=\sum_{k\in\Z\setminus\br{0}}\frac{\abs{\widehat{f^\prime}(k)}^2}{k^2}\leq \sum_{k\in\Z\setminus\br{0}}\abs{\widehat{f^\prime}(k)}^2=\sum_{k\in\Z}\abs{\widehat{f^\prime}(k)}^2=\norm{f^\prime}^2,$$
completing the proof.
\end{proof}

\bigskip
\begin{proof}[Proof of Lemma \ref{lem:properties_of_antiderivative}]
Since for any function $h\in L^1\pa{\T}$
$$\pa{h-h_{\mathrm{avg}}}_{\mathrm{avg}}=0\ ,$$
we conclude \eqref{item:avg} from the definition of $\partial_x^{-1}f(x)$. To show \eqref{item:diff} we invoke the fundamental theorem of calculus (the version from Lebesgue theory), and to show \eqref{item:anti_of_deri} we notice that if $f$ is differentiable
$$\partial^{-1}_x\pa{\partial_x f}(x)=\int_{0}^{x} \partial_y f(y)dy-\frac{1}{2\pi}\int_{0}^{2\pi} \pa{\int_{0}^{x} \partial_y f(y)dy}dx$$
$$=f(x)-f(0)-\frac{1}{2\pi}\int_{0}^{2\pi} \pa{f(x)-f(0)}dx=f(x)-f_{\mathrm{avg}}. $$
Lastly, we notice that the continuity of $\partial_x^{-1}f(x)$ as a function on the interval $[0,2\pi]$ is a standard result from Analysis. To conclude the continuity on the torus, though, we must also show that $\partial_x^{-1}f (0)=\partial_x^{-1} f\pa{2\pi}$. This is equivalent to
$$0=\int_{0}^{0}f(x)dx = \int_{0}^{2\pi}f(x)dx=2\pi f_{\mathrm{avg}},$$
which is exactly the additional assumption. In addition, \eqref{eq:fourier_coefficients_of_antiderivative} for $k\not= 0$ follows immediately from \eqref{eq:fourier_of_derivative} and \eqref{item:diff}. For $k=0$ we use
$$\widehat{\partial^{-1}_x f}\pa{0}=\pa{\partial^{-1}_x f}_{\mathrm{avg}}=0,$$
according to \eqref{item:avg}. The proof is thus complete.
\end{proof}

\bigskip
\begin{proof}[Proof of Lemma \ref{lem:entropy_equivalence}]
We will establish that 
$$\abs{\frac{\theta}{2\pi}\int_{0}^{2\pi} \partial_x^{-1}f (x)\overline{g(x)}dx} \leq \frac{\abs{\theta}}{2}\pa{\norm{f}^2+\norm{g}^2} $$
from which \eqref{eq:entropy_equivalence_always}, \eqref{eq:entropy_equivalence_restricted} and \eqref{eq:entropy_equivalence} all follow. Indeed, using the Cauchy-Schwarz inequality, the Poincar\'e inequality -- which is valid since $(\partial^{-1}_x f)_{\mathrm{avg}}=0$ - and part \eqref{item:diff} of Lemma \ref{lem:properties_of_antiderivative} we conclude that 
$$\abs{\frac{\theta}{2\pi}\int_{0}^{2\pi} \partial_x^{-1}f (x)\overline{g(x)}dx} \leq \abs{\theta}\norm{\partial_x^{-1}f}\norm{g}\leq \frac{\abs{\theta}}{2}\pa{\norm{\partial_x^{-1}f}^2+\norm{g}^2}$$
$$  \leq \frac{\abs{\theta}}{2}\pa{\norm{\partial_x\pa{\partial_x^{-1}f}}^2+\norm{g}^2}=\frac{\abs{\theta}}{2}\pa{\norm{f}^2+\norm{g}^2}\ .$$
The proof is thus completed. 
\end{proof}

\begin{proof}[Proof of \eqref{item:max_less_min_gamma} and \eqref{item:max_less_theta} within the proof of Lemma \ref{lem:value_of_alpha}]
To show \eqref{item:max_less_min_gamma} we notice that $\gamma_{\mathrm{max}}\pa{\theta} \leq \gamma_{\mathrm{min}}\pa{\theta}$ if and only if
$$\frac{2\sqrt{4-\theta^2}+\pa{4-\smax\theta}}{2+\sqrt{4-\theta^2}} \leq \frac{2\sqrt{4-\theta^2}-\pa{4-\smin \theta}}{2-\sqrt{4-\theta^2}}$$
which is equivalent to
$$ 2\pa{8-\theta\pa{\smin+\smax}}+ \sqrt{4-\theta^2}\;\theta\pa{\smax-\smin}\leq 4\pa{4-\theta^2},$$
or
$$ \frac{2\pa{\smax+\smin - 2\theta}}{\smax-\smin} \geq \sqrt{4-\theta^2}.$$
The above inequality is easily satisfied when $\theta \leq \min\pa{\smin,2}$ as in this case
$$\frac{2\pa{\smax+\smin - 2\theta}}{\smax-\smin} \geq  \frac{2\pa{\smax-\smin}}{\smax-\smin}=2\geq \sqrt{4-\theta^2}.$$\\ 
\eqref{item:max_less_theta} is trivial, since $\theta<\smax$ holds if and only if $\gamma_{\mathrm{max}}(\theta)<\theta$.
\end{proof}

\bibliographystyle{spmpsci} 

\end{document}